\documentclass{amsart}
\usepackage[colorlinks=true]{hyperref}
\usepackage{mathrsfs}
\usepackage{booktabs}
\usepackage{mathtools}
\usepackage{geometry}
\usepackage{setspace}
\usepackage[english]{babel}
\usepackage[T1]{fontenc}
\usepackage[utf8]{inputenc}
\usepackage{enumitem}
\usepackage{graphicx}
\usepackage{fancyhdr}
\usepackage{amsmath}
\usepackage{amsfonts}
\usepackage{amssymb}
\usepackage{csquotes} 
\usepackage{amsthm}
\usepackage{geometry,rotating,tabstackengine}
\usepackage[arrow, matrix, curve]{xy}
\usepackage{tikz}
\usepackage{pgfplots}
\pgfplotsset{compat=1.18} 
\usepackage{amsfonts}
\usepackage{bbm}
\usepackage{algorithm}
\usepackage{algorithmic}           
\usepackage{listings}
\usepackage{longtable}
\usepackage{comment}
\usetikzlibrary{shapes,positioning}
\newcommand{\CK}{\mathbf{G}}
\newcommand{\A}{\mathcal{A}}
\theoremstyle{plain}
\newtheorem{theorem}{Theorem}[section]

\newtheorem{corollary}[theorem]{Corollary}

\newtheorem{lemma}[theorem]{Lemma}

\theoremstyle{definition}
\newtheorem{definition}[theorem]{Definition}

\newtheorem{proposition}[theorem]{Proposition}
\theoremstyle{remark}
\newtheorem{remark}[theorem]{Remark}

\newtheorem*{remark*}{Remark}
\newtheorem*{example*}{Example}
\newtheorem{example}[theorem]{Example}

\newcommand{\rom}[1]{\uppercase\expandafter{\romannumeral #1\relax}}

\renewcommand{\epsilon}{\varepsilon}
\newcommand{\rank}{\operatorname{rank}}

\newcommand{\codim}{\operatorname{codim}}

\newcommand\CA{{\mathcal A}}

\newcommand\CE{{\mathscr E}}

\newtheorem{innercustomgeneric}{\customgenericname}
\providecommand{\customgenericname}{}
\newcommand{\newcustomtheorem}[2]{%
  \newenvironment{#1}[1]
  {%
   \renewcommand\customgenericname{#2}%
   \renewcommand\theinnercustomgeneric{##1}%
   \innercustomgeneric
  }
  {\endinnercustomgeneric}
}

\newcommand\restr[2]{{
  \left.\kern-\nulldelimiterspace
  #1 
  \littletaller
  \right|_{#2}
  }}

\newcommand\restrtwo[2]{%
  \left.\kern-\nulldelimiterspace
  #1
  \right|_{#2}%
}

\newcustomtheorem{customthm}{Theorem}
\newcustomtheorem{customlemma}{Lemma}

\title{Boolean building set arrangements}

\author{Leonie Mühlherr}
\address{Fakult\"at f\"ur Mathematik,
	Universit\"at Kassel, D-34132 Kassel, Germany}
\email{leonie.muehlherr@mathematik.uni-kassel.de}
\author{Sven Wiesner}
\address{Fakultät für Mathematik, Ruhr-Universität Bochum, D-44780 Bochum, Germany}
\email{sven.wiesner@rub.de}
\usepackage[backend=biber]{biblatex}
\begin{document}
\begin{abstract}
    In this article, we generalize the notion of connected subgraph arrangements, established by Cuntz and Kühne by allowing edges of cardinality greater than two, with a focus on local properties. We show that the class of free connected hypersubgraph arrangements is larger than the class of free connected subgraph arrangements, but that there exists a necessary condition for freeness that refers back to the graphic case. Furthermore, we obtain a unique description of these arrangements by associating them with Boolean building sets. We study the nested complex of the building set in relation to the lattice of intersections, and characterize all simplicial arrangements within this new class. 
 
\end{abstract}
\maketitle
\section{Introduction}

Let $V$ be a vector space of dimension $n$ over $\mathbb{Q}$. Let $x_1,\dots, x_n$ be the dual basis of $V^*$. We introduce the main characters of this article: 

\begin{definition}\label{definition: connected hypersubgraph arrangements}
    Let $G=(N,E)$ be a simple hypergraph with vertex set $N=[n]=\{1,2,\dots, n\}$ and edge set $E\subset 2^N$. We define the \emph{connected hypersubgraph arrangement} $\CA_G$ in $V=\mathbb{Q}^n$ as
\[\CA_G:=\{H_I\mid \emptyset\neq I, G[I] \text{ is connected} \},\]
where $H_I$ is the hyperplane $H_I= \ker\sum_{i\in I}x_i$.
\end{definition} 

This definition is motivated by a similar construction in \cite{cuntzkuehne:subgrapharrangements}, where $\CA_G$ is defined for $G$ a simple graph. From this definition, well-known arrangements such as the braid arrangement and the resonance arrangement can be recovered. Moreover, Cuntz and Kühne achieve a complete characterization of freeness (see Theorem \ref{theorem: FreeConnectedSubgraphArrangements}) and simpliciality of connected subgraph arrangements by combining tools from graph and arrangement theory. Since then, connected subgraph arrangements gained significant attention, as they were studied with a focus on related or distinct properties \cite{freemultiderivationsconnectedsubgraph, connectedsubgrapharrangements}.

As is the case for connected subgraph arrangements, every connected hypersubgraph arrangement is a \emph{$(0/1)$-arrangement}, i.e., we can choose a normal vector for each hyperplane that contains only zeros and ones.  

When generalizing the setting to hypergraphs, two arrangements defined via different hypergraphs may yield the same arrangement (see Example \ref{ex:build}). We solve this problem by  considering them under a new combinatorial angle: a connected hypersubgraph arrangement can be naturally associated to a Boolean building set. 

\begin{lemma}\cite[Lemma~3.9]{feichtnersturmfels:matroidpolybergmanfans} \label{lem:booleanbuilding sets}
    A family $\mathcal{F}$ of subsets of $[n]=\{1,2,\dots,n\}$ is a building set in the Boolean lattice $2^{[n]}$ if
and only if $\mathcal{F}$ contains all singletons $\{i\}, i \in [n]$, and the following condition holds: if
$F, F' \in \mathcal{F}$ and $F \cap F'\neq \emptyset$ then $F \cup F' \in \mathcal{F}$.
\end{lemma}

We prove the following: 
\begin{lemma} \label{lem:hypergrahs and building sets}
    Let $\mathcal{A}_G$ be a connected hypersubgraph arrangement and denote by $\mathcal{F}_{\mathcal{A}_G}$ the set of index sets $I$ such that $H_I \in \mathcal{A}_G$. Then $\mathcal{F}_{\mathcal{A}_G}$ is a building set in the Boolean lattice.  
\end{lemma}

In the case of simple graphs, this set is known as the \emph{graphical building set} (see \cite{postnikov:permutohedra}, Example 7.2).
If two hypergraphs produce the same arrangement, their underlying building sets coincide. 
Building sets were introduced by Feichtner and Kozlov \cite{feichtnerkozlov:incidence} generalizing a concept used by De Concini and Procesi for hyperplane arrangement complement compactification \cite{deconciniprocesi:wonderful}. They arise in the context of polytopes and fans \cite{feichtnersturmfels:matroidpolybergmanfans,postnikov:permutohedra},  as well as in convex geometry \cite{backmanndanner:convexgeometry}. 

In this article, we focus on the nested complex of a building set (see Definition \ref{def:nest_set}), analyse the rank of the corresponding intersections in $L(\mathcal{A})$, and compute the $b_2$ coefficient of such arrangements. Furthermore, we showcase a connection between our arrangement and a special class of polytopes associated with building sets, so-called nestohedra (see Subsection \ref {subsec:nesto}).

Since this article generalizes the graph case to hypergraphs, its goal is to broaden the results achieved in \cite{cuntzkuehne:subgrapharrangements} by recovering the employed tools in a more general setting. 
The main freeness result pertains to the underlying connected subgraph arrangement $\A_{\underline{G}}$ (see Definition \ref{Definition: graphs created by 2 building sets}) and is analogous to the result in the graphic case: 
\begin{theorem}\label{theorem: A_B2 is free}
    Let $G=(N,E)$ be a hypergraph. If $\CA_G$ is free, then $\CA_{\underline{G}}$ is free. 
\end{theorem}

This is a necessary but not sufficient condition for an arrangement to have the freeness property, see Example \ref{ex:freeness_condition} for non-free arrangements with a free underlying connected hypersubgraph arrangement. However, the class of simplicial connected hypersubgraph arrangements coincides with the arrangements defined on simple graphs: 

\begin{theorem}\label{thm:simpliciality}
 Let $G=(N,E)$ be a hypergraph. Then $\CA_G$ is simplicial if and only if $\CA_G$ is a connected subgraph arrangement  with $G=C_3$ or $G=P_n$, that is, if $\mathcal{F}_{\mathcal{A}_G} = \mathcal{F}_{\mathcal{A}_{C_3}}$ or $\mathcal{F}_{\mathcal{A}_G} = \mathcal{F}_{\mathcal{A}_{P_n}}$. 
\end{theorem}

\medskip 
This article is organized as follows: In Section \ref{sec: preliminaries} we recall basic definitions of hyperplane arrangement theory and the previous work of Cuntz and Kühne, as it pertains to our generalizations. In Section \ref{sec:building sets} we discuss a characterization of these new arrangements via Boolean building sets and analyse the nested complex. In Section \ref{sec:hypersubgraphs}, we record first observations on connected hypersubgraph arrangements, which we use to define the underlying connected subgraph arrangement and develop tools to reduce the complexity of arrangements in Section \ref{sec:operations}. In Section \ref{sec:freeness}, we prove results for the class of free arrangements, in Section \ref{sec:simplicial} we discuss simplicial arrangements. Examples obtained through computations for small dimensions can be found in Appendix \ref{sec:appendix}. 

\subsection{Acknowledgements} The authors would like to thank Lukas Kühne and Gerhard Röhrle for helpful discussions and feedback and an anonymous referee for comments that helped improve this manuscript and for pointing us to connections mentioned in Section \ref{subsec:nesto}. 
LM was funded by the Deutsche Forschungsgemeinschaft (DFG, German Research Foundation) –
	Project-ID 491392403 – TRR 358 and supported by German Research Foundation grant 522790373. 
	
\subsection{Use of AI declaration} 

We used AI to correct mistakes in English writing.

\section{Preliminaries} \label{sec: preliminaries}

In this section, we give a brief introduction to hyperplane arrangements and some associated objects and properties. Furthermore, we summarize results obtained by Cuntz and Kühne \cite{cuntzkuehne:subgrapharrangements} in order to motivate the main construction of this paper. 
\subsection{Hyperplane arrangements}
\subsubsection*{Basic definitions}
This section closely follows \cite{orlikterao:arrangements}.
\begin{definition} 
	Let $\mathbb{K}$ be a field and let $V$  be a vector space of dimension $n$ over $\mathbb{K}$. A \emph{hyperplane} $H$ in $V$ is a subspace of dimension $n-1$. A hyperplane arrangement $\mathcal{A}$  is a finite set of hyperplanes in $V$. 
\end{definition} 
\begin{remark}
    In general, hyperplanes can also be affine spaces; all arrangements we consider in this work are the so-called \emph{central} arrangements defined above. 
\end{remark}

\noindent
Let $V^*$ be the dual space of $V$ and $S = S(V^*)$ be the symmetric algebra of $V^*$. Fix a basis $\{x_1,x_2,\dots,x_n\}$ for $V^*$ and identify $S$ with the polynomial algebra $S = \mathbb{K}[x_1,\dots,x_n]$. 
\begin{definition} 
	Let $\mathcal{A}$ be a hyperplane arrangement. Each hyperplane $H \in \mathcal{A}$ is the kernel of a linear polynomial $\alpha_H$ defined up to a constant. The product $Q(\mathcal{A}) = \prod_{H \in \mathcal{A}} \alpha_H$ is called a \emph{defining polynomial} of $\mathcal{A}$. 
\end{definition} 
\noindent
The polynomial $\alpha_H$ is a linear form for all $H$, and $Q(\mathcal{A})$ is a homogeneous polynomial. 

A prominent problem in the theory of hyperplane arrangements is understanding the invariants and characteristic features of arrangements through their combinatorial structure: 
	\begin{definition} 
	Let $L(\mathcal{A})$ be the set of all non-empty intersections of hyperplanes of $\mathcal{A}$ \[L(\mathcal{A}) = \big\{ \textstyle\bigcap_{H\in \mathcal{A'}}H\mid \mathcal{A'} \subseteq \mathcal{A}\big\}\]
	For $X\in L(\A)$ and $n\geq k\in\mathbb{Z}_{\geq0}$, define $\rank(X)=\codim_V X$ and  $\rank(\A)=\rank(\cap_{H\in\A}H)$. 
\end{definition}

\medskip 
Defining an ordering on $L(\mathcal{A})$ by reverse inclusion ($X \leq  Y \Leftrightarrow Y \subseteq X$) gives us a geometric lattice. $L(\mathcal{A})$ is called the \emph{intersection lattice}. 

\begin{definition}\label{def:charac_pol} 
	Define the \emph{Möbius function} $\mu_\CA = \mu: L(\CA) \times L(\CA) \rightarrow \mathbb{Z}$ through the following equations \begin{align*} 
		\mu(X,X) = 1  & \text{ if $X\in L(\CA)$,} \\ 
		\sum_{X \leq Z \leq Y} \mu(X,Z) = 0 & \text{ if $X,Y,Z \in L(\CA)$ and $X < Y$,} \\ 
		\mu(X,Y) = 0 &\text{ otherwise.}
	\end{align*} 
	
	Define the \emph{characteristic polynomial} of $\CA$ as $ \chi(\CA, t) = \sum_{X\in L(\CA)} \mu(V,X)\cdot t^{\dim X}$, and its \emph{Poincaré polynomial} as $\pi(\A, t) = \sum_{X\in L(\CA)} \mu(X) (-t)^{\rank(X)}$. We denote the coefficients of the Poincaré polynomial as $b_i$. 
	
\end{definition}

An important tool to understand an arrangement is to look at its behavior under deletion and restriction: 

	For $X\in L(\mathcal{A})$, define the \emph{localization} $\mathcal{A}_X$ of $\mathcal{A}$ by $\mathcal{A}_X = \{H \in \mathcal{A}\mid X \subseteq H\}$, as well as $(\mathcal{A}^X, X)$, an arrangement in $X$, by $\mathcal{A}^X = \{X \cap H\mid H \in \mathcal{A}\setminus\mathcal{A}_X~\text{and}~X\cap H \not= \emptyset\}$ called the \emph{restriction} of $\mathcal{A}$ to $X$. 
	
	\begin{definition} 
		Let $\mathcal{A}$ be a non-empty arrangement and let $H_0 \in \mathcal{A}$. Let $\mathcal{A'} = \mathcal{A}\setminus\{H_0\}$ and let $\mathcal{A''} = \mathcal{A}^{H_0}$. We call $(\mathcal{A}, \mathcal{A'}, \mathcal{A''})$ a \emph{triple of arrangements with distinguished hyperplane $H_0$}. 
	\end{definition} 
    Localization is useful to us because some of the properties we are interested in are local: 
    \begin{definition}
        	A property $P$ of an arrangement $\mathcal{A}$ is called \emph{local} if it is preserved under localizations, that is, if $\mathcal{A}$ has property $P$, then $\mathcal{A}_X$ has $P$ for all $X\in L(\mathcal{A})$. 
    \end{definition}
    Thus, to prove that an arrangement does not have a local property $P$, it suffices to find one specific localization for which $P$ does not hold. 
    
    The following lemma classifies rank $2$ localizations for $(0/1)$-arrangements.
    \begin{lemma}\cite[Lemma~2.10]{cuntzkuehne:subgrapharrangements}\label{lemma: connected subgraph arrangements are locallyA2}
    	For pairwise distinct nonempty subsets $A_1, A_2, A_3\subseteq  [n]$, and $n\geq1$ the following two conditions are equivalent:
    	\begin{itemize}
    		\item[(i)] $\rank(H_{A_1}\cap H_{A_2}\cap H_{A_3})=2$ and
    		\item[(ii)] $A_{i_1}\dot{\cup}A_{i_2} = A_{i_3}$ for some choice of pairwise different indices $1 \leq i_1, i_2, i_3 \leq 3$ where $\dot{\cup}$
    		denotes a disjoint set union.
    	\end{itemize}
    \end{lemma}
    We can associate a module to an arrangement $\mathcal{A}$: 
\begin{definition} 
	A $\mathbb{K}$-linear map $\theta: S \rightarrow S$ is a derivation if for $f,g \in S$: \[\theta(f\cdot g) = f\cdot \theta(g)+g\cdot \theta(f).\] Let $\text{Der}_{\mathbb{K}}(S)$ be the \emph{module of derivations of $S$}. \\ Define an submodule of $\text{Der}_{\mathbb{K}}(S)$, called the \emph{module of $\mathcal{A}$-derivations}, by \[D(\mathcal{A}) =\bigcap_{H\in \mathcal{A}} D(\alpha_H) = \{\theta \in \text{Der}_{\mathbb{K}}(S)\mid \theta(\alpha_H)\in \alpha_H S~\text{for all}~H\in \mathcal{A}\}.\]
	The arrangement $\mathcal{A}$ is called \emph{free} if $D(\mathcal{A})$ is a free $S$-module. If $\theta_1,\dots,\theta_n$ form a basis of $D(\A)$, then the set $\exp \A := \{\deg \theta_1, \dots, \deg \theta_n\}$ is called the exponents of $\A$. 
\end{definition} 
\begin{example} 
The \emph{Euler derivation} $\sum_{i = 1}^n x_i D_i$ is in $D(\A)$ for any arrangement $\A$. It is usually denoted by $\theta_E$.
\end{example} 
\begin{definition} \label{def:matcoef} 
			Given derivations $\theta_1,\dots,\theta_n\in D(\mathcal{A})$, define the coefficient matrix $M(\theta_1,\dots,\theta_n)$ as $M_{i,j} = \theta_i(x_j)$.
		\end{definition} 
		\begin{theorem}\cite[Theorem 4.19]{orlikterao:arrangements}(Saito's criterion) \label{thm:saito}

			Given $\theta_1,\dots,\theta_n\in D(\mathcal{A})$, the following two conditions are equivalent: 
			\begin{enumerate} 
				\item $\det M(\theta_1,\dots,\theta_n) = c\cdot Q(\mathcal{A})$ for some $c\in \mathbb{K}^*$. 
				\item $\theta_1,\dots,\theta_n$ form a basis for $D(\mathcal{A})$ over  $S$. 
			\end{enumerate} 
		\end{theorem} 
\begin{definition}
Let $\text{rank}(\CA)> k \in\mathbb{Z}_{\geq 0}$. We call an arrangement $\A$ \emph{generic} if for all $X\in L(\A)_k$ it holds that $\vert\A_X \vert = k$. 
\end{definition} 

\begin{definition}
	We denote by $M(\CA)=V\setminus (\cup_{H\in\CA} H)$ the complement of $\CA$ in $V$.
An arrangement is called $K(\pi, 1)$ if its complement is a $K(\pi, 1)$-space. 
\end{definition}

\begin{definition} \label{def:simp}
    Let $n\in\mathbb{N}, V:=\mathbb{R}^n$, and $\CA$ an arrangement in $V$. Let $\mathcal{K}(\CA)$ be the set of connected components (chambers) of $V\setminus\cup_{H\in\CA}H$. If every chamber $K$ is an open simplicial cone, i.e., there exist $\alpha_1^\lor,\dots, \alpha_n^\lor\in V$ such that $$K=\left\{\sum_{i=1}^n a_i\alpha_i^\lor\mid a_i>0\text{ for all } i=1,\dots,n\right\}=:\langle \alpha_1^\lor,\dots, \alpha_n^\lor\rangle_{>0},$$ then $\CA$ is called a \emph{simplicial arrangement}.
\end{definition}
\subsubsection*{Important freeness results}

The first result we want to mention is a classic result which considers the arrangement triples $(\mathcal{A}, \mathcal{A'}, \mathcal{A''})$ to establish freeness criteria for a given arrangement: 
\begin{theorem}\label{thm:add-del}(Addition-Deletion-Theorem) \cite[Theorem~4.51]{orlikterao:arrangements}
    Suppose $\CA \not= \Phi_n$. Let $(\mathcal{A}, \mathcal{A'}, \mathcal{A''})$ be a triple of arrangements. Any two of the following statements imply the third: 
    \begin{enumerate}
        \item $\CA$ is free with $\exp \CA = \{e_1,\dots, e_{n-1}, e_n\}$. 
        \item $\CA'$ is free with $\exp \CA' = \{e_1, \dots,e_{n-1}, e_n -1 \}$. 
        \item $\CA''$ is free with $\exp \CA'' = \{e_1,\dots, e_{n-1}\}$. 
    \end{enumerate}
\end{theorem}

The next theorem is stated in a more general form for multiarrangements in \cite{yoshinaga:extendable}; we thus define them here. 

\begin{definition}\cite{ziegler}
	A \emph{multiarrangement} is a pair $(\CA, m)$ consisting of a hyperplane arrangement $\CA$ and a multiplicity function $m: \CA \rightarrow \mathbb{Z}_{> 0}$ that associates with each hyperplane $H$ in $\CA$ a non-negative integer $m(H)$. 
\end{definition}

We can define the freeness of a multiarrangement $(\CA, m)$ as the freeness of the analogous derivation module \[D(\CA, m) := \{\theta \in \text{Der}(S) \mid \theta(\alpha_H) \in \alpha_H^{m(H)}S \text{ for each } H \in \CA\}.\]
An arrangement $\CA$ is called \emph{totally non-free} if $(\CA ,m)$ fails to be free for any multiplicity $m$, with $m(H) \geq 1$ for all $H\in \CA$.

\begin{theorem}\cite[Proposition~4.1]{yoshinaga:extendable}\label{theorem: Yoshinaga generic not free}
   Let $\CA$ be a generic arrangement in $V$ with $n=\dim V\geq 3$. If $\vert\CA\vert>n$, then $(\CA, m)$ is not free for any $m: \CA \rightarrow \mathbb{Z}_{> 0}$.
\end{theorem}

\begin{theorem}\cite[Theorem~4.37]{orlikterao:arrangements}\label{theorem: localizations are free}
	If $\CA$ is free, then $\CA_X$ is free for all $X\in L(\CA)$. 
\end{theorem}

\subsection{Connected subgraph arrangements}
The following class of arrangements was introduced by Cuntz and Kühne.
\begin{definition}\cite[Definition~1.1]{cuntzkuehne:subgrapharrangements}
        Let $G=(N, E)$ be an undirected simple graph with a vertex set $N=\{1,\dots,n\}$ and edge set $E$. We define the \emph{connected subgraph arrangement} $\CA_G$ in $V=\mathbb{Q}^n$ as
$\CA_G:=\{H_I\mid \emptyset\neq I\subseteq N \text{ and } G[I] \text{ is connected}\}$,
where $H_I$ is the hyperplane $H_I= \ker\sum_{i\in I}x_i$ and $G[I]$ is the induced subgraph on the vertex set $I\subseteq N$. 
\end{definition}

Their main result is a classification of all graphs $G$ such that $\CA_G$ is a free arrangement.
\begin{theorem}\cite[Theorem~1.6]{cuntzkuehne:subgrapharrangements}\label{theorem: FreeConnectedSubgraphArrangements}
Let $G$ be a connected graph. The connected subgraph arrangement $\CA_G$ is free if and only if $G$ is either

\begin{enumerate} 
	\item a path-graph $P_n$, where $n$ is the number of vertices,
	\item a cycle-graph $C_n$, where $n$ is the number of vertices, 
	\item an almost-path-graph $A_n,k$; a path on $n$ vertices with an additional edge $\{n+1, k\}$, or 
	\item a path-with-triangle-graph $\Delta_{n,k}$; a path on $n$ vertices with two additional edges $\{k, n+1\}, \{k+1, n+1\}$.
\end{enumerate} 
\end{theorem}

See Figure \ref{fig:freeness_graphs} for example graphs. 

	\begin{figure}[H]
			\includegraphics[width=0.7
			\textwidth]{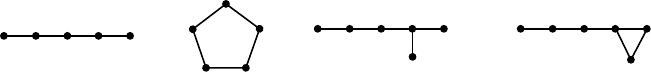} 
			\caption{Small example graphs of free arrangements} 
			\label{fig:freeness_graphs}
		\end{figure}

\noindent They showed that every other connected subgraph arrangement has a non-free localization and therefore fails to be free. In this classification, the following graphs played a key role. In \cite{cuntzkuehne:subgrapharrangements} they were named $G_1$ to $G_8$, thus to avoid confusion with our hypergraph notation, we denote them as $\CK_1$ to $\CK_8$.

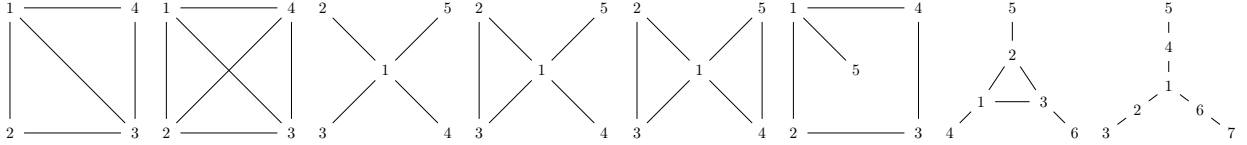
\begin{figure}[H]
	\centering
\resizebox{\textwidth}{!}{
\begin{tikzpicture}[scale=.7,auto=left,every node/.style={circle,fill=white!10}]
  \node (a1) at (1,5) {1};
  \node (a2) at (1,1)  {2};
  \node (a3) at (5,1)  {3};
  \node (a4) at (5,5)  {4};
  \foreach \from/\to in {a1/a2,a2/a3,a3/a4,a1/a4,a1/a3}
    \draw (\from) -- (\to);

  \node (b1) at (6,5) {1};
  \node (b2) at (6,1)  {2};
  \node (b3) at (10,1)  {3};
  \node (b4) at (10,5)  {4};
  \foreach \from/\to in {b1/b2,b2/b3,b3/b4,b1/b4,b1/b3,b2/b4}
    \draw (\from) -- (\to);

  \node (c1) at (13,3) {1};
  \node (c2) at (11,5) {2};
  \node (c3) at (11,1) {3};
  \node (c4) at (15,1) {4};
  \node (c5) at (15,5) {5};
    
  \foreach \from/\to in {c1/c2,c1/c3,c1/c4,c1/c5}
    \draw (\from) -- (\to);

  \node (d1) at (18,3) {1};
  \node (d2) at (16,5) {2};
  \node (d3) at (16,1) {3};
  \node (d4) at (20,1) {4};
  \node (d5) at (20,5) {5};  
  
  \foreach \from/\to in {d1/d2,d1/d3,d1/d4,d1/d5,d2/d3}
    \draw (\from) -- (\to);
\end{tikzpicture}

\begin{tikzpicture}[scale=.7,auto=left,every node/.style={circle,fill=white!10}]\label{Graphs: 8 not free graphs}
  
  \node (a1) at (3,3) {1};
  \node (a2) at (1,5) {2};
  \node (a3) at (1,1) {3};
  \node (a4) at (5,1) {4};
  \node (a5) at (5,5) {5};  
  
  \foreach \from/\to in {a1/a2,a1/a3,a1/a4,a1/a5,a2/a3,a4/a5}
    \draw (\from) -- (\to);

  \node (b1) at (8,3) {5};
  \node (b2) at (6,5) {1};
  \node (b3) at (6,1) {2};
  \node (b4) at (10,1) {3};
  \node (b5) at (10,5) {4};  
  
  \foreach \from/\to in {b1/b2,b2/b3,b3/b4,b4/b5,b2/b5}
    \draw (\from) -- (\to);  

  \node (c1) at (11,1) {4};
  \node (c2) at (12,2) {1};
  \node (c3) at (13,3.5) {2};
  \node (c4) at (13,5) {5};
  \node (c5) at (14,2) {3};
  \node (c6) at (15,1) {6};
   
  \foreach \from/\to in {c1/c2,c2/c3,c3/c5,c3/c4,c2/c5,c5/c6}
    \draw (\from) -- (\to);  

  \node (d3) at (16,1) {3};
  \node (d2) at (17,1.75) {2};
  \node (d4) at (18,3.75) {4};
  \node (d5) at (18,5) {5};
  \node (d6) at (19,1.75) {6};
  \node (d7) at (20,1) {7};
  \node (d1) at (18,2.5) {1};
   
  \foreach \from/\to in {d1/d2,d2/d3,d1/d4,d4/d5,d1/d6,d6/d7}
    \draw (\from) -- (\to);   
\end{tikzpicture}} \caption{The graphs $\CK_1$ up to $\CK_8$.}\label{Figure: G1 to G8}
\end{figure}

\section{Building sets and equivalence of hypergraphs} \label{sec:building sets}

This section connects a subclass of building sets to the connected hypersubgraph arrangements. This is motivated by the fact that two distinct hypergraphs do not necessarily give distinct connected hypersubgraph arrangements. 

\begin{example} \label{ex:build}
    Let $G_1=([3],\{\{1,2\},\{2,3\},\{1,2,3\}\}), G_2=([3],\{\{1,2\},\{2,3\}\})$. Then 
    $$\CA_{G_1}=\CA_{G_2}=\{H_{1},H_{2},H_{3},H_{12},H_{23},H_{123}\}.$$
\end{example}

We can solve this problem by first associating Boolean building sets to hypergraphs:  

\begin{definition}\cite[Definition~3.1]{feichtnersturmfels:matroidpolybergmanfans} Let $\mathcal{L}$ be a finite lattice. A subset $\mathcal{G}$ of $\mathcal{L}>\hat{0}$ is a building set if for any
$X \in \mathcal{L}>\hat{0}$ and $\max \mathcal{G}_{\leq X} = \{g_1 , \dots, g_k \}$ there is an isomorphism of partially ordered sets
$$\varphi_X :\Pi_{j=1}^{k} [\hat{0},g_j]\to[0,\hat{X}],$$
where the j-th component of the map $\varphi_X$ is the inclusion of intervals $[\hat{0},g_j] \subset [\hat{0}, X]$
in $\mathcal{L}$.    
\end{definition}

If $\mathcal{L}$ is a Boolean lattice, then Lemma \ref{lem:booleanbuilding sets} can be used to classify building sets instead.

\begin{lemma}\label{lem: intersection_boolean_sets}
    The intersection of two building sets in the Boolean lattice is again a building set in the Boolean lattice.
\end{lemma}
\begin{proof}
    Let $B_1, B_2$ be building sets on the ground set $G=[n]$. \\
    Since $B_1$ and $B_2$ are building sets all singletons $\{i\}$ for an arbitrary $i\in G$ are contained in the intersection $B_1\cap B_2$. Now let $b_1,b_2\in B_1\cap B_2$ with $b_1\cap b_2\neq\emptyset$. Since $B_1, B_2$ are building sets, we have that  $b_1\cup b_2 \in  B_i, i=1,2$. This shows that $B_1\cap B_2$ is a building set on the ground set $G$.
\end{proof}

We prove Lemma \ref{lem:hypergrahs and building sets}. 
\begin{proof}[Proof of Lemma \ref{lem:hypergrahs and building sets}]
    From Remark \ref{remark: connected hypersubgraph arrangement remarks} (1), we know that all singletons are contained in $\mathcal{F}_{\mathcal{A}_G}$. Now take $F, F' \in \mathcal{F}_{\mathcal{A}_G}$ with $F \cap F'\neq \emptyset$. Take two vertices $u,v \in G[F\cup F']$ and select a vertex $w \in F \cap F'$. By definition of $\mathcal{A}_G$, there are paths $ue_1v_2e_2\dots e_kw$ and $we_{k+1}v_{k+2}e_{k+2}\dots e_sv$ in $G$. We get a path from $u$ to $v$ by concatenating these and thus $G[F\cup F']$ is connected, so $F\cup F' \in \mathcal{F}_{\mathcal{A}_G}$.
\end{proof}
Note that if the hypergraph $G$ on $n$ vertices is connected (see Section \ref{sec:hypersubgraphs}), the building set $\mathcal{F}_{\CA_G}$ is also connected, i.e., $[n] \in \mathcal{F}_G$. 

Feichtner and Sturmfels introduced the notion of building closure. 
\begin{definition}\cite[Lemma 3.10]{feichtnersturmfels:matroidpolybergmanfans}
    Let $N=[n]$ and $E\subseteq 2^N$. We denote by $\hat{\mathcal{F}}(E)$ the unique minimal building set generated by $E$, i.e., the set $\mathcal{F}_E$ of minimal cardinality such that $E \subseteq \mathcal{F}_E$ and $\mathcal{F}_E$ is a building set. Call $\hat{\mathcal{F}}(E)$ the \emph{building closure} of $E$. 
\end{definition}

\begin{lemma}
    Let $G = (N, E)$ be a hypergraph. Then $\mathcal{F}_{\A_G} = \hat{\mathcal{F}}(E)$. 
\end{lemma}
\begin{proof} 
The inclusion $\mathcal{F}_{\A_G} \supseteq \hat{\mathcal{F}}(E)$ is clear, since $E \subseteq \mathcal{F}_{\A_G}$. 
Let $I \in \mathcal{F}_{\A_G}$, but not contained in $\hat{\mathcal{F}}(E)$, then one of the following holds:
\begin{enumerate} 
\item  $I \in E(G)$.
\item  There exist $e_1,e_2,...,e_k \in E(G)$ s.t. $(\cup_{1\leq i\leq k} e_i)=I$ and for all $1\leq i < k$ it holds that $e_i \cap e_{i+1} \neq \emptyset$.

\end{enumerate} 
 If (1), then $I\in \hat{\mathcal{F}}(E)$ since $E(G)\subseteq \hat{\mathcal{F}}(E)$. If (2), then $I\in \hat{\mathcal{F}}(E)$ by Lemma 1.2. Since $I$ was arbitrary, this holds for any $I\in \mathcal{F}_{\A_G}$, so $\mathcal{F}_{\A_G}\subseteq \hat{\mathcal{F}}(E)$ and finally $\mathcal{F}_{\A_G}=\hat{\mathcal{F}}(E)$. 
\end{proof}
\begin{definition}
    Say two hypergraphs $G_1, G_2$ are equivalent if $\mathcal{F}_{\mathcal{A}_{G_1}} = \mathcal{F}_{\mathcal{A}_{G_2}}$. For each equivalence class, choose as representative a hypergraph $G$ with $E(G) = \mathcal{F}_{\A_G}$, i.e., the hypergraph with edge set precisely the sets of the building set of the arrangement. 
\end{definition}

\begin{remark}
	\begin{enumerate}
	\item The idea of considering graphs with respect to a special edge set is similar to the concept of atomic, saturated, connected (ASC) graphs introduced by Do\v{s}en and Petri\'c in \cite{dosenpetric}, as their underlying edge set is also a building set in the Boolean lattice. 
	\item The relation of hypergraphs and building sets was used by Jeli\'c, Jevti\'c and \v{Z}ivaljevi\'c in their study of the connection of cyclohedra (a special type of nestohedron coming from the building set associated with the path graph, see Section \ref{subsec:nesto} for details on this construction) to Kantorovich-Rubinstein polytopes \cite{jevticjeliczivaljevic}. They also use the linear forms we define our hyperplanes with and establish restriction and deletion for building sets.  
    \item Even though the building set $\hat{\mathcal{F}}(E)$ is minimal among building sets containing $E$, the representative we chose for each class does not correspond to the hypergraph with minimal edge set in this class. In Example \ref{ex:build}, the representative of the equivalence class shared by the two hypergraphs $G_1, G_2$ is $G_1$, the larger of the two. In the following, we will write $\mathcal{A}_{\mathcal{F}}$ for such an arrangement, with the understanding that we can always interpret the building set as a set of hyperedges and thus return to the graph theoretical interpretation. 
    \end{enumerate}
\end{remark}

The next section makes a connection between the nested complex of a building set and the intersection lattice of the corresponding arrangement. 
\begin{definition}\cite[Definition~7.3]{postnikov:permutohedra} \label{def:nest_set}
Let $\mathcal{F}$ be a building set. A \emph{nested set} $\mathfrak{n}$ in $\mathcal{F}$ is a collection of non-empty subsets of $\mathcal{F}$ such that
    \begin{itemize}
        \item for any $n_1,n_2\in \mathfrak{n}$, we have either $n_1\subseteq n_2, n_1\supseteq n_2$, or $n_1\cap n_2=\emptyset$.
        \item For any collection of $k \geq 2$ disjoint subsets $n_1,n_2,\dots, n_k \in \mathfrak{n}$, their union $n_1\cup n_2\cup\dots \cup n_k$ is not in $\mathcal{F}$.
        \item $\mathfrak{n}$ contains all inclusion-wise maximal elements of $\mathcal{F}$.
    \end{itemize}

\begin{remark}\label{rem:max_set} 
    We only consider irreducible connected hypersubgraph arrangements, so the building sets are connected (see Lemma \ref{lemma: G not in B then AB reducible}). Thus, there is only one inclusion-wise maximal element, and it is the vertex set of the hypergraph. In this case, the third condition imposed on nested sets is just that this element is contained in all of them. 
\end{remark}
\end{definition}
The \emph{nested complex} $\mathcal{N}(\mathcal{F})$ is the poset of all nested sets of $\mathcal{F}$, ordered by inclusion.

\subsection{Rank of intersections for nested sets}  
We examine the rank of the intersections of hyperplanes defined by nested sets. 

\begin{lemma} \label{lem:nest_set}
    Let $\mathcal{F}_{\mathcal{A}_G}$ be a building set of a connected hypersubgraph arrangement and $\mathfrak{n}$ a nested set of $\mathcal{F}_{\mathcal{A}_G}$ with $\vert \mathfrak{n} \vert = k$. Then the intersection \[X_{\mathfrak{n}} = \bigcap_{I \in \mathfrak{n}} H_I\] has rank $k$ in $L(\mathcal{A}_G)$, i.e., full rank. 
\end{lemma}
\begin{proof}
Order the elements $I_1,\dots, I_k$ of $\mathfrak{n}$ by increasing cardinality. We want to show that there is an ordering on the sets such that each $I_j$ contains an element not in $\bigcup_{i = 1}^{j-1} I_i$. Assume that at some point $j$ in the sequence, we have $I_j \subseteq \bigcup_{i = 1}^{j-1} I_i$.

Because of the ordering and the nested set property, we have that all $I_i, i < j$ with non-trivial intersection are contained in $I_j$, thus $I_j = \bigcup_{\substack{i < j \\ I_i \cap I_j \not= \emptyset }} I_i$. If two of these sets have a non-trivial intersection, then one already is contained in the other. This implies that the union $\bigcup_{\substack{i < j \\ I_i \cap I_j \not= \emptyset }} I_i$, i.e. $I_j$, can be written as a union of disjoint sets, which contradicts the nested set property. 

Now consider the submatrix of the defining matrix of $X_\mathfrak{n}$ consisting of the rows and columns corresponding to the indices $i_j, 1\leq j \leq k$. With some row and column operations, this becomes an upper triangular matrix with ones on the diagonal, thus it must have rank $k$. 
\end{proof}

\noindent
So each nested set corresponds to an intersection of full rank, but this is not uniquely determined.

\begin{example}
   Consider the arrangement $\mathcal{A}_{P_4}$. 
   
   Then, the nested sets $\{\{1\},\{1,2\}, \{1,2,3,4\}\} \text{ and } \{\{2\},\{1,2\}, \{1,2,3,4\}\}$ define the same intersection of hyperplanes. 
\end{example}
Define an operation on nested sets in the following way: 
\begin{definition}
    Let $\mathcal{F}$ be a building set. Call a set $C \subset \mathcal{F}, \vert C\vert = k$ a \emph{$k$-circuit} of $\mathcal{F}$, if there exists
    an element $b_C \in C$ such that $\bigsqcup_{I \in C\setminus b_c} I = b_C$.
    
Call $\vert b_C \vert$ the size of $C$.
We call a subset $B_C\subsetneq C$ with cardinality $\vert C\vert -1$ 
    an \emph{almost-$k$-circuit}.  
\end{definition}
In \cite{cuntzkuehne:subgrapharrangements}, the case for $k = 3$ was discussed under the name \emph{circuit-triple}. 
By definition, the nested sets of a building set cannot contain any circuits. 
\begin{lemma} \label{lem:elmt_exchange}
    Let $\mathfrak{n}$ be a nested set of a building set $\mathcal{F}$ that contains an almost-circuit $B_C$. Define the element $\overline{c} := C \setminus B_C$. Then for any non-maximal element $c \in B_C$, there exists an element $c' \in \mathcal{F}$ such that $\overline{c} \subseteq c'\subseteq b_C$ and the set $\mathfrak{n}\setminus \{c\} \cup  \{c'\}$ is again a nested set. 
\end{lemma}
\begin{proof}
    We want to switch out the element $c$ for $\overline{c}$, while preserving the nested set condition. If there is another almost-circuit $B_{C'}$ that would be completed by $\overline{c}$, then $\overline{c} \not= b_{C'}$. We know that $b_C \cap b_{C'} \not= \emptyset$ (because the intersection contains $\overline{c}$), and thus a containment relation must hold. We conclude that $b_{C'} \subsetneq b_C$; otherwise, by the first condition of nested sets, there would already be a set of sets in $\mathfrak{n}$ whose disjoint union would yield $b_{C'}$ (because $b_C \in \mathfrak{n}$).
    
    Let $C'$ be the size-wise largest circuit in $\mathcal{F}$ such that $b_{C'} \subsetneq b_C$, $\overline{c} \in C'$, and $ c \notin C'$. If $C'$ exists, then define $c' = b_{C'}$ ($b_{C'}$ is not yet in $\mathfrak{n}$, since otherwise there would be a complete circuit contained in $\mathfrak{n}$), if it does not exist, define $c' = \overline{c}$. 
    We need to check the conditions of Definition \ref{def:nest_set} for this new set. The third condition follows immediately from the choice of $c$. Removing $c$ does not affect the first two conditions. We need to check that they still hold after adding $c'$.  
    
    \begin{enumerate} 
    \item[(1)] Assume there exists an element $n \in \mathfrak{n}$ such that $c'\cap n \not= \emptyset$, this implies that $n\cap b_C \not= \emptyset$. If $b_C \subseteq n$, then $c'\subseteq n$. If $n \subseteq b_C$, assume that $n \not\subseteq c'$, then there is an element $a \in n \setminus c'$. It immediately follows  that $a\in b_C$ and we conclude that there exists another set $c'' \in C$ for which $n \cap c'' \not= \emptyset$, implying that the condition would have been violated before adding $\overline{c}$. 
    \item[(2)] The second condition is fulfilled by the choice of $c'$: if there is a circuit in $\mathfrak{n}\setminus \{c\} \cup  \{\overline{c}\}$, of which $\overline{c}$ is part, then we chose it as the $b_{C'}$ of the largest such circuit. Assuming there were a circuit with $b_{C'}$ in this setting, this circuit would be comprised of elements other than $\overline{c}$ and hence $\mathfrak{n}$ was never a nested set. 
    \end{enumerate}
\end{proof}
\begin{example}
    Consider the building set associated with $P_3$: $\mathcal{F} = \{\{1\}, \{2\}, \{3\}, \{1,2\}, \{2,3\}, \{1,2,3\}\}$.

    The set $U = \{\{1\}, \{2\}, \{3\}, \{1,2,3\}\}$ is a 4-circuit of size 3. The set $\mathfrak{n} = \{\{1\}, \{3\}, \{1,2,3\}\}$ is a nested set and an almost-circuit. Let $c = \{3\}$, then we have that $\mathfrak{n} \setminus \{\{3\}\} \cup \{\{2\}\}$ is not a nested set, because of the circuit $ \{\{1\}, \{2\}, \{1,2\}\}$ contained in $\mathcal{F}$, therefore we choose $c' = \{1,2\}$ and the set $\{\{1\}, \{1,2\}, \{1,2,3\}\}$ is again a nested set.  
\end{example}

The third condition of the nested set definition restricts the type of hyperplane intersections we can obtain. Define the  \emph{generalized nested complex} as the complex of all sets fulfilling the first two conditions of Definition \ref{def:nest_set}. Note that we did not use the third condition in the proof of Lemma \ref{lem:nest_set}, so it also holds for generalized nested sets. We get a generalized version of Lemma \ref{lem:elmt_exchange} by dropping the assumption that $c$ cannot be maximal.

\begin{proposition}
    If a nested set $M$ of a building set $\mathcal{F}$ is obtained by applying Lemma \ref{lem:elmt_exchange} to a nested set $M'$ of $\mathcal{F}$, then both define the same intersection of hyperplanes in $\A_{\mathcal{F}}$. 
    
\end{proposition}
\begin{proof}
The intersections of hyperplanes $H$ are the kernels of matrices with $\alpha_H$ as row vectors. To prove that two intersections are the same is to prove that the two matrices have the same kernel. 
Consider an application of Lemma \ref{lem:elmt_exchange} to a nested set with an almost-circuit. In the notation of the lemma, adding the rows corresponding to elements in $B_C\setminus \{b_c\}$ and the row corresponding to $b_c$ (if contained in $B_C$) with different sign yields the row vector corresponding to $\overline{c}$ up to sign. Adding this row and deleting another corresponding to some element $c \in B_C$ would only change the kernel of this matrix if there is another row relation when adding the new row. This is exactly the circuit condition in the proof of Lemma \ref{lem:elmt_exchange}. If such a circuit $C'$ exists, then adding the row corresponding to $b_{C'}$ creates a row combination to again obtain the row corresponding to $\overline{c}$ and by consequence by another linear combination the row corresponding to $c$. As a result, applying the lemma to a nested set will not change its corresponding hyperplane intersection. 
\end{proof} 

\begin{remark}
    Not all full-dimensional intersections of hyperplanes can be represented by nested sets, and thus, sets yielding the same intersection are not always related by Lemma \ref{lem:elmt_exchange}. 
    
    Consider the building set \[\mathcal{F} = \{\{1\}, \{2\},\{3\},\{4\},\{5\},\{6\},\{1, 2,3,4\}, \{1,2,5,6\}, \{3,4,5,6\}, \{1,2,3,4,5,6\}\}\] and the two sets \[M = \{\{1,2,5,6\}, \{3,4,5,6\}, \{1,2,3,4,5,6\}\}~\text{and}~M' = \{\{1,2,3,4\}, \{3,4,5,6\}, \{1,2,3,4,5,6\}\}\] 
    Both corresponding matrices have rank 3 and the same kernel. 
        
        Neither set contains any almost-circuits (since $\mathcal{F}$ does not contain any 2-element sets) and hence, Lemma \ref{lem:elmt_exchange} is not applicable. 
\end{remark}

In general, describing the structure of the lattice of intersections remains difficult, but in the case of two-dimensional intersections, we can conclude the following: 
\begin{corollary} \label{cor_b_2} 
    Let $\mathcal{F}$ be a building set and let $c_3$ be the number of 3-circuits in $\mathcal{F}$. Then, we can compute the $b_2$-coefficient of the Poincaré polynomial of the corresponding arrangement $\mathcal{A}_{\mathcal{F}}$ as 
    \[b_2 = \binom{\vert \mathcal{F}\vert}{2} - c_3. \]
\end{corollary}
\begin{proof} 
The $b_2$ coefficient is the sum of all Möbius function values of rank 2 intersections in $L(\A)$. Two 2-element sets have the same intersection if they contain an almost-circuit to the same circuit. Since the smallest circuit has 3 elements, there cannot be an equivalence class of size more than 3. The Möbius function can therefore admit two values.  
\begin{enumerate}
    \item If there is just one element in the equivalence class, then the value of the Möbius function is 1. 
    \item If there is an equivalence class corresponding to a circuit, the value of the Möbius function is 2. 
\end{enumerate}
Adding all the values, we get the term above. 
\end{proof} 

\begin{remark} 
	Terao's Factorization Theorem states that if an arrangement is free, then its Poincaré polynomial factors as $\pi(\CA, t) = \prod_{k \in \exp \CA}  (1+ kt)$. It is also well-known that the sum of all exponents equals the number of hyperplanes in the arrangement (see \cite{orlikterao:arrangements} for both statements). This means that Corollary \ref{cor_b_2} may be used to prove that an arrangement cannot factor, if the value of $b_2$ is too large to be calculated from any assumed factorization with coefficients summing up to the number of hyperplanes. 
\end{remark} 
\subsection{A note on nestohedra} \label{subsec:nesto} 
Given a building set $\mathcal{F}$, there is a polytope associated with it: 

\begin{definition}\cite{postnikovreinerwilliams}
	Let $\mathcal{F}$ be a building set on $[n]$, and for $I \in \mathcal{F}$, define $\Delta_I$ as the convex hull of $\{e_i\mid i \in I\}$, where $e_i$ is the standard basis vector in $\mathbb{R}^n$. Define the \emph{nestohedron} $P_{\mathcal{F}}$ as the Minkowski sum: \[P_{\mathcal{F}} := \sum_{I \in \mathcal{F}} y_I \cdot \Delta_I,\] with $y_I$ strictly positive real parameters. 
\end{definition} 

This is also known in the literature as a certain type of \emph{generalized permutahedron} \cite{postnikov:permutohedra} or \emph{nested polytope} \cite{zelevinsky}, as it is strongly related to the nested complex of $\mathcal{F}$: 

\begin{theorem}\cite[Theorem 7.4 \& Proposition 7.5]{postnikov:permutohedra}
	The lattice of faces of $P_{\mathcal{F}}$ (ordered by reverse inclusion) is isomorphic to $\mathcal{N}(\mathcal{F})$.
	
	Furthermore, let $\mathcal{F}$ be a building set and $N\in \mathcal{N}(\mathcal{F})$ a nested set. The face associated with $N$ is given by \[P_N = \{(t_1, \dots , t_n) \in \mathbb{R}^n \mid \sum_{i \in I} t_i = z_I \text{ for $I \in N$ and } \sum_{i \in J} t_i \geq z_J, \text{ for }J \in \mathcal{F}\}\] for $z_I := \sum_{J \subseteq I} y_J$. The dimension of $P_N$ is $n - \vert N \vert$. 
\end{theorem}

In our setting, since the building set is connected, there is just one set included in all nested sets, namely $[n]$, and thus the polytope $P_{\mathcal{F}}$ has dimension $n-1$. Hence, the facets are characterized by the equations $\sum_{i \in I} t_i = z_I$ for all $I \in \mathcal{F}$ (since $\{I, [n]\}$ is always a nested set and all nested sets of size 2 are of this form). If we consider the non-affine versions of these facet-defining hyperplanes, we recover the restriction of $\mathcal{A}_{\mathcal{F}}$ to the hyperplane $H_{[n]}$. 

\section{Connected hypersubgraph arrangements} 

\subsection{Preliminaries}\label{sec:hypersubgraphs}
Our goal is to generalize the connected subgraph arrangements by extending the graph class to hypergraphs. We collect some basic hypergraph definitions given in \cite{bahmanianSajna:connectedhypergraphs}.
\begin{definition}
A \emph{hypergraph} $G$ is an ordered pair $(N,E)$, where $N$ and $E$ are disjoint finite sets such that $N\neq\emptyset$ and there is a function $\psi:E\to 2^N$, called the \emph{incidence function}. We call the set $N=N(G)$ the vertex set and $E=E(N)$ the edge set of $G$. 
\end{definition}
\begin{definition}
Two edges $e,e'\in E$ are parallel if $\psi(e)=\psi(e')$. A hypergraph is called \emph{simple} if $\psi$ is injective.\\
In this article, we examine simple hypergraphs $G=(N,E)$, where (for some $n\in\mathbb{N}$) we have $N=[n]:=\{1,2,\dots,n\}$ and $E\subset 2^{[n]}$. Because of this, we omit the incidence function $\psi$ and write edges explicitly. 

To distinguish between edges of graphs and hypergraphs, we call an edge $e\in 2^{[n]}$ a \emph{genuine hyperedge} if $\vert e\vert\geq 3$, a hyperedge if $\vert e\vert\geq 2$ and an edge if $\vert e\vert= 2$.
\end{definition}

The concept of paths and by extension connectivity of a graph can be extended to hypergraphs in a natural way: 
\begin{definition}
 Call $u,v\in N$ \emph{adjacent} if there exists an $e\in E$ with $\{u,v\}\subseteq e$. 

Let $0\leq k\in \mathbb{Z}$. A \emph{$(u,v)$-walk} of length $k$ in $G$ is a sequence $v_1e_1v_2e_2\dots e_kv_{k+1}$ with $v_1,v_2,\dots,v_{k+1}\in N$, $e_1,e_2,\dots,e_k\in E$, $v_1=u$, $v_{k+1}=v$, and $v_i,v_{i+1}$ are adjacent in $G$ via the edge $e_i$ for all $1 \leq i \leq k$.

A $(u,v)$-walk of length $k$ is called a \emph{path} if all $k+1$ vertices $v_i$ and all $k$ edges $e_i$ are distinct.
\end{definition}
 
 We call $G$ \emph{connected} if for all $u,v\in N$ there exists a $(u,v)$-path in $G$.
Let $I\subseteq N$. Define the \emph{induced hypersubgraph of $G$}  on $I$ as $G[I]:=(I,\{e\in E\mid e\subseteq I\})$.

\medskip 
We note some facts about connected hypersubgraph arrangements (see Definition \ref{definition: connected hypersubgraph arrangements}). For readability, we will often denote a hyperplane $H_{i_1,\dots, i_k}$ as $H_{i_1\dots i_k}$.

\begin{remark}\label{remark: connected hypersubgraph arrangement remarks}
    \begin{enumerate}
        \item[(1)] for all $i\in N$ we have $H_i = H_{\{i\}}\in\CA_G$ since $G[i]$ is connected (use a path of length $0$).
        \item[(2)] Every connected hypersubgraph arrangement $\CA_G$ is a subarrangement of the connected subgraph arrangement $\CA_{K_n}$, where $K_n$ is the complete graph with $n$ vertices. In particular, we have $\vert\CA_X\vert\leq 3$ for all $X\in L(\CA)$ with $\rank(X)=2$ {\cite[Lemma~2.10]{cuntzkuehne:subgrapharrangements}}.
        \item[(3)] If for $G=(N,E)$ we have $\vert e\vert=2$ for all $e \in E$, then the connected hypersubgraph arrangement is just a connected subgraph arrangement.
        \item[(4)] Not every connected hypersubgraph arrangement is a connected subgraph arrangement.\\ Example: Define the hypergraph $G=(\{1,2,3\},\{\{1,2,3\}\})$, then $\CA_G=\{H_1,H_2,H_3,H_{123}\}$ is a subarrangement of the connected subgraph arrangement $\CA_{K_3}$, where $K_3$ is the complete graph on $3$ vertices, but there does not exist a graph $\mathcal{G}$ such that $\CA_G=\CA_{\mathcal{G}}$.
        \item[(5)] Not every free connected hypersubgraph arrangement is a connected subgraph arrangement. See Table \ref{tab:building sets_4} for free, non-graphic examples with up to 4 vertices. 
\end{enumerate}
\end{remark}

\subsection{Tools for hypersubgraph arrangements} \label{sec:operations}

 We start by investigating the graph of simple edges of $G$ and its connected subgraph arrangement to see if this already restricts the local properties that the connected hypersubgraph arrangement $\CA_G$ can have.  

\begin{definition}\label{Definition: graphs created by 2 building sets}
	For a hypergraph $G=(N,E)$ we define $\underline{E}=\{e\in E\mid \vert e\vert=2\} \text{ and the graph }\underline{G}=(N,E_2).$ Call the corresponding connected hypersubgraph arrangement $\CA_{\underline{G}}$ the \emph{underlying connected subgraph arrangement of $\CA_G$}. In particular, we have a decomposition of $\CA_{\underline{G}}$ into connected subgraph arrangements $\CA_{\underline{G}_i}$ as  $\CA_{\underline{G}} = \CA_{\underline{G}_1}\times\CA_{\underline{G}_2}\times\dots\times\CA_{\underline{G}_k}$, which is a subarrangement of $\CA_G$. We write $\mathscr{C}_G=\{\underline{G}_1,\underline{G}_2,\dots,\underline{G}_k\}$ for the set of connected components of $\underline{G}$. 
\end{definition}   

\begin{remark}\label{Remark: Underlying Connected Subgraph Remarks}
	Let $G=(N,E)$ be a hypergraph. 
	\begin{enumerate}
		\item For any $\underline{G}_i=(N_i,\underline{E}_i)\in\mathscr{C}_G$, the induced hypersubgraph $G[N_i]$ contains the subgraph $\underline{G}_i$ and the connected hypersubgraph arrangement $\CA_{G[N_i]}$ lies between the connected hypersubgraph arrangements $\CA_{\underline{G}_i}$ and $\CA_{K_n}$, where $K_n$ is the complete graph on $n$ vertices. 
		\item Let $e\in E$ such that $e\cap N_i\neq \emptyset$ for at least two $\underline{G}_i=(N_i,\underline{E}_i)$, then $\vert e\vert \geq 3$. Otherwise, the $\underline{G}_i$ would be connected by the edge $e$.
		If instead $e\subseteq N_i$ and $H_e\not\in\CA_{\underline{G}_i}$, then $\underline{G}_i[e]$ is not connected. Otherwise, by definition, we would already have $H_e\in \CA_{\underline{G}_i}$. 
	\end{enumerate}
\end{remark}

The underlying connected subgraph arrangement is of great importance as we showcase in the next section.

\bigskip
\noindent 
In \cite{cuntzkuehne:subgrapharrangements}, the authors developed tools to work with connected subgraph arrangements by relating graph operations such as contraction and taking induced subgraphs to arrangement theory. Since hypergraphs are a more general case, we need to reexamine this strategy to see to what extent we may apply the same techniques. 
In this section, we generalize the key tools used by Cuntz and Kühne to the hypergraph case. 
\begin{lemma}\label{lemma: G not in B then AB reducible}
    Let $G=(N,E)$ be a hypergraph, then $\CA_G$ is irreducible if and only if $G$ is connected.
\end{lemma}
\begin{proof}
     If $G$ is connected, then we have $H_N\in\CA_G$ (since $G=G[N]$), and therefore $\CA_G$ contains the irreducible subarrangement $\{H_1,H_2,\dots, H_n, H_N\}$. So $\CA_G$ is irreducible.\\
    Conversely, if $G$ is not connected, we can write $N=I_1\dot{\cup}I_2\dot{\cup}\dots\dot{\cup}I_k$ with $k\geq 2$, where $I_i\subseteq N$ such that the $G[I_i]$ are the maximal connected components of $G$. Since $I_i\cap I_j=\emptyset$ for $i\neq j$, we have $$\CA_G\cong\CA_{G[I_1]}\times\CA_{G[I_2]}\times\dots\times\CA_{G[I_k]},$$ so $\CA_G$ is reducible. \qedhere
\end{proof}

\noindent As in the case for simple graphs, there is a connection between the connected hypersubgraph arrangements of a hypergraph $G=(N,E)$ and an arbitrary induced hypersubgraph of $G$ on a subset of vertices $M\subseteq N$. The following proof is analogous to its simple graph version in \cite{cuntzkuehne:subgrapharrangements}. 

\begin{lemma}\label{lemma: Induced subgraph is localization}
    Let $G=(N,E)$ be a hypergraph on $n$ vertices, let $\CA_G$ be the corresponding connected hypersubgraph arrangement, and $M\subseteq N$. If $\CA_G$ has a local property $P$, then $\CA_{G[M]}$ has $P$ as well.
\end{lemma}
\begin{proof}
    By definition, $G[M]$ is the hypergraph with vertex set $M$ and edge set $\{e\in E\mid e\subseteq M\}$. 
     Now consider $X=\bigcap_{j\in M}H_{\{j\}}\in L(\CA_G)$. We show that $(\CA_G)_X= \CA_{G[M]}$
    when embedding $\CA_{G[M]}$ into $\mathbb{Q}^n$.
    Let $H_I\in\CA_G$ with $I\subseteq N$ arbitrary, then by definition, $H_I\in (\CA_G)_X$ if and only if $X\subseteq H_I$. But this is equivalent to $I\subseteq M$ (otherwise choose an $i\in I, i\not\in M$ and $x=(x_1,\dots,x_n)$ with $x_i=1, x_j = 0  ~(j\neq i)$. Then $x$ lies in $X$ but not in $H_I$ which contradicts $X\subseteq H_I$).
    By definition, this shows that $G[I]$ as a subgraph of $G$ is equal to $G[I]$ as a subgraph of $G[M]$. We know that $H_I\in\CA_G$ and therefore $G[I]$ is connected, so $H_I\in\CA_{G[M]}$. This shows $(\CA_G)_X \subseteq \CA_{G[M]}$. If $H_I\in\CA_{G[M]}$, then $I\subseteq M$. Since for $(x_1,x_2,\dots,x_n)\in X$ it holds that $x_i=0$ if $i\in M$ we have $x\in H_I$. This proves the equality above.  \\
    Since $P$ is a local property and $\CA_G$ has property $P$, $\CA_{G[M]}$ as its localization has property $P$ as well.
\end{proof}

We generalize the concept of edge contractions from the graph case to the hypergraph case.

\begin{definition}\label{definition: edge contraction}
    Let $G=(N,E)$ be a hypergraph and $e\in E$.  We define $G/e$ the contraction of $G$ by $e$ as the following hypergraph. Fix some $i\in e$  
    and define $N'=(N\setminus e)\cup\{i\}$ as the ground set for $G/e$. Define the new set of edges as $$E'=\{e'\in E\mid e'\cap e=\emptyset\}\cup\{e'\setminus e\cup\{i\}\mid e'\cap e\neq\emptyset\}.$$ 
    Note that $G/e$ is the connected hypersubgraph arrangement defined by the hypergraph $G/e=(N',E')$. In other words: All vertices $j \in e$ are identified as one fixed vertex $i\in e$ and the set of neighbors of $i$ is the union of the sets of neighbors of all vertices $j\in e$. 
\end{definition}

\noindent
As a convention, we choose $i$ to be the smallest vertex label of $e$. 

\noindent
The proof of the next statement also follows the proof of its graphic counterpart \cite[Lemma~6.4]{cuntzkuehne:subgrapharrangements} closely. 

\begin{lemma}\label{lemma: contraction of building set}
    Let $G=(N,E)$ be a hypergraph on $n$ vertices, let $\CA_G$ be its connected hypersubgraph arrangement, and $e\in E$. If $\CA_G$ has a local property $P$, then $\CA_{G/e}$ has the property $P$.
\end{lemma}

\begin{proof}
    Our goal is to find an intersection $X\in L(\A)$, such that we can express the contraction along an edge as a localization, which proves the claim. 
    
    Let $e\in E$, fix $i\in e$ and let $G/e=(N',E')$ as in Definition \ref{definition: edge contraction}. We identify all vertices in $e$ with the vertex $i$ after contracting $G$ by $e$ and relabel the vertices such that $N'=\{1,2,\dots,i\}$, so $e=\{i,i+1,\dots,n\}$. Define $X=H_e \cap \big(\bigcap_{1\leq k\leq i-1} H_k \big) $. We construct an embedding $\varphi$ from $(\mathbb{Q}^{i})^*$ to $(\mathbb{Q}^n)^*$ showing $\CA_{G/e}\simeq(\CA_G)_X$. Define the linear map $\varphi:(\mathbb{Q}^{i})^*\to(\mathbb{Q}^n)^*$ by  $$\varphi(x_k) = \begin{cases} x_k & k < i\\ \sum_{j=i}^n  x_j & k = i. \end{cases} $$
    Note that since $\varphi$ is linear, it induces a map between the hyperplanes $\ker(\alpha)$ and $\ker(\varphi(\alpha))$ for all $\alpha\in(\mathbb{Q}^i)^*$. First, we show that $\varphi$ is well-defined by proving $\varphi(\alpha_H)(x)=0$ for all $H\in\CA_{G/e}, x\in X$. Let $J\subseteq N'$ such that $H_J\in\CA_{G/e}$, then we have one of the two following cases.
    \begin{enumerate}
    \item[$i\in J$]: Because $i\in J$, we have $\varphi(\alpha_{H_J})=\alpha_{H_{J\cup e}}$. The induced subgraph $(G/e)[J]$ is connected by assumption. The equalities $J\cup e=N$ and $J\cap e=\{i\}$ imply that the induced subgraph $G[J\cup e]$ is connected, which shows $H_{J\cup e}\in\CA_G$.\\
    It remains to show that $X\subseteq H_{J\cup e}$. Choose an arbitrary $x=(x_1,x_2,\dots,x_n)\in X$, then $$x_k=0, (\text{for }k\leq i-1)\text{ and }\sum_{k=i}^n x_k=0,$$ which shows that $x\in H_{J\cup e}$ and therefore $X\subseteq H_{J\cup e}$.    
    \item[$i\not\in J$]: In this case, $J\subseteq \{1,2,\dots,i-1\}\text{ and }\varphi(\alpha_{H_J})=\alpha_{H_J}.$ In particular, $G[J]=(G/e)[J]$ is connected. Let $x=(x_1,x_2,\dots,x_n)\in X$, then $x_k=0$ for $k< i$, so $x\in H_J$ which shows that $X\subseteq H_J$. \end{enumerate}
    \noindent 
    We conclude that $\varphi$ is well-defined and injective. If $H_J\in(\CA_G)_X$, then $\alpha_{H_J}\in \text{span}\{x_1,x_2,\dots,x_{i-1},\sum_{k=i}^n x_k\}$. So there exist $a_1,a_2,\dots,a_i\in \{0,1\}$ such that $$ \alpha_{H_J}=\sum_{k=1}^{i-1}a_kx_k+a_i\sum_{k=i}^nx_k,$$ which gives $\varphi^{-1}(\alpha_{H_J})=\sum_{k=1}^i a_kx_k$ by definition of $\varphi$. Analogous to the two cases above, one confirms $\ker(\varphi^{-1}(\alpha_{H_J}))\in \CA_{G/e}$ and therefore $\varphi$ induces a surjective map between $\CA_{G/e}$ and $(\CA_G)_X$.    
\end{proof}

\section{Freeness} \label{sec:freeness}
Let $G$ be a connected hypergraph with underlying simple graph $\underline{G}$. In this section, we show that given a free connected hypersubgraph arrangement $\CA_G$ the underlying connected subgraph arrangement $\CA_{\underline{G}}$ must also be free (see Theorem \ref{theorem: A_B2 is free}).

 We then showcase some obstructions to freeness of connected hypersubgraph arrangements between a free connected subgraph arrangement and the connected subgraph arrangement of the complete graph.
 
 Lastly, we give a simple construction for obtaining new free arrangements. 

\subsection{Proof of Theorem \ref{theorem: A_B2 is free}}
The proof is as follows: In Lemmas \ref{lemma: G1G2 subgraph totally non-free} to \ref{lemma: G8 subgraph totally non-free}, we consider the graphs $\CK_i, 1\leq i \leq 8$ in Figure \ref{Figure: G1 to G8} and show that there are no free connected hypersubgraph arrangements between $\CA_{\CK_i}$ and $\CA_{K_{n_i}}$, where $n_i = \vert V(\CK_i)\vert$. Some of these partial results are broader in scope, proving totally non-freeness or a general statement for 0/1-arrangements.

 Combining these findings, we conclude in Proposition \ref{Proposition: underlying connected subgraph arrangment is free} that the only possible choices for the connected components of $\underline{G}$ are the four classes of graphs listed in Theorem \ref{theorem: FreeConnectedSubgraphArrangements}. From Proposition \ref{Proposition: underlying connected subgraph arrangment is free} and Theorem \ref{theorem: FreeConnectedSubgraphArrangements} we deduce Theorem \ref{theorem: A_B2 is free} without further calculations.

\noindent In this section, we number the vertices of graphs $\CK_1$ to $\CK_8$ as in Figure \ref{Figure: G1 to G8} and create the corresponding connected hypersubgraph arrangements $\mathcal{A}_{\CK_i}$ according to Definition \ref{definition: connected hypersubgraph arrangements}. We showcase the idea of proof for Lemmas \ref{lemma: G1G2 subgraph totally non-free} to \ref{lemma: G8 subgraph totally non-free} in the following remark.

\begin{remark}
    Let $G=(N,E)$ be a hypergraph. First, assume that there exists a $\underline{G}_i=(N_i,\underline{E}_i)\in\mathscr{C}_G$ (recall Definition \ref{Definition: graphs created by 2 building sets}) such that $\underline{G}_i$ has some graph $\CK_j, j\in\{1,2,\dots, 8\}$, see Figure \ref{Figure: G1 to G8} or $\CK_j$ with additional edges attached as a subgraph. Take $N_{j}\subseteq N_i$ such that $G[N_{j}]$ is equal to $\CK_j$ or to $\CK_j$ with additional (hyper-)edges attached. According to Remark \ref{Remark: Underlying Connected Subgraph Remarks} $(1)$, the connected hypersubgraph arrangement $\CA_{G[N_{j}]}$ lies between $\CA_{\CK_j}$ and $\CA_{K_{n_j}}$ (where $n_j=\rank\CA_{\CK_j}$). By Lemma \ref{lemma: Induced subgraph is localization}, that arrangement is a localization of $\CA_G$ and by Theorem \ref{theorem: localizations are free}, freeness is a local property. So if $\CA_G$ is free, then $\CA_{G[N_{j}]}$ must be free. However, we show that no free connected hypersubgraph arrangements exist between $\CA_{\CK_j}$ (for $j\in\{1,2,\dots,8\}$) and $\CA_{K_n}$. In particular, $\CA_{G[N_{j}]}$ fails to be free and consequently, so does $\CA_G$. This shows that $\CA_G$ cannot be free if a $\underline{G}_i\in\mathscr{C}_G$ contains some $\CK_j$ as a subgraph.\\
\end{remark}

\begin{lemma}
    Let $G=(N,E)$ be a hypergraph and $\mathscr{C}_G$ as in Definition \ref{Definition: graphs created by 2 building sets}. If one of the $\underline{G}_i$ is equal to $\CK_1$ or $\CK_2$, then $\CA_G$ is not free.
\end{lemma}\label{lemma: G1G2 subgraph totally non-free}
\begin{proof}
    Let $\underline{G}_i\in\mathscr{C}_G$ be equal to $\CK_1$ or $\CK_2$, and let $N_{a_i}\subseteq N$ be the corresponding elements of the ground set $N$. Taking the induced subgraph of $G$ on the vertex set $N_{a_i}$ yields precisely the connected subgraph arrangement of $\CK_1$ or $\CK_2$ (since $\CA_{\CK_2}$ is the connected subgraph arrangement coming from $K_4$ and $\CA_{\CK_1}$ is a deletion of it), hence $\CA_G$ fails to be free by Lemma \ref{lemma: Induced subgraph is localization}.
\end{proof}

\begin{lemma}\label{lemma: no free between G3,G4,G5,G6 and K_5}
    Let $\underline{G}_i = \CK_i$ for $i \in \{3,4,5,6\}$. There are no free $(0/1)$-arrangements between $\CA_{\underline{G}_i}$ and $\CA_{K_5}$.
\end{lemma}
\begin{proof}
Since $\CK_4$ and $\CK_5$ have $\CK_3$ as a subgraph and consequently $\CA_{\CK_3}$ is a subarrangement of $\CA_{\CK_4}$ and $\CA_{\CK_5}$, it is sufficient to consider $\CA_{\CK_3}$.\\
Let $X=(H_{124}\cap H_{125}\cap H_{134})\in L(\CA_{\CK_3})$, with $(\CA_{K_5})_X=(\CA_{\CK_3})_X=\{H_{124}, H_{125}, H_{134}, H_{135}\}$. This localization is a generic arrangement of rank $3$ in $\CA_{K_5}$. Therefore, the localization is totally non-free by Theorem \ref{theorem: Yoshinaga generic not free}. However, this arrangement is also a rank $3$ localization of $\CA_{\CK_3}$, since all hyperplanes are already included in $\CA_{\CK_3}$. This implies that every subarrangement of $\CA_{K_5}$ containing $\CA_{\CK_3}$ as a subarrangement contains this localization and is totally non-free since freeness is a local property.

Let $G=\CK_6$ and $X=(H_{123}\cap H_{125}\cap H_{134})\in L(\CA_{\CK_6})$. Then we get \[(\CA_{K_5})_X=(\CA_{\CK_6})_X=\{H_{123}, H_{125}, H_{134}, H_{145}\}\] which is a generic arrangement of rank $3$, allowing us to use the same argument as for $\CA_{\CK_3}$.
 \end{proof}

\begin{lemma}\label{lemma: G7 subgraph totally non-free}
    Let $G=(N,E)$ be a hypergraph and $\mathscr{C}_G$ as in Definition \ref{Definition: graphs created by 2 building sets}. If one of the $\underline{G}_i$ has $\CK_7$ as a subgraph, then $\CA_G$ is totally non-free.
\end{lemma}

\begin{proof}
    By Lemma \ref{lemma: Induced subgraph is localization}, it is sufficient to choose $N_1\subseteq N$ such that $G[N_1]$ has $\CK_7$ as a subgraph and prove the statement for $G=G[N_1]$. Assume that $\CA_G$ is free.
    Let $X=(H_{23}\cap  H_{136}\cap H_{1235}\cap H_{1245})\in L(\CA_{\CK_7}),$ then we get the localization
    $(\CA_{\CK_7})_X=\{H_{23}, H_{136}, H_{1235}, H_{1245}, H_{12346}\},$ which is a generic arrangement of rank $4$. The localization of $\CA_{K_6}$ at $X$ is $$(\CA_{K_6})_X=\{H_{15}, H_{23}, H_{24}, H_{136}, H_{146}, H_{1235}, H_{1245}, H_{12346}\}$$
    and we compute $(\CA_{K_6})_X\setminus(\CA_{\CK_7})_X=\{H_{15}, H_{24}, H_{146}\}$.
    If $\CA_G$ is free, then at least one of the elements of $(\CA_{K_6})_X\setminus(\CA_{\CK_7})_X$ has to be included in $\CA_G$ (otherwise the existence of $\CA_X$ contradicts the freeness of $\CA_G$). However, if one of the edges $\{1,5\}$ or $\{2,4\}$ is part of $G$, then $G$ has a vertex $v$ with at least four neighbors. Taking the induced subgraph on $v$ and four of its neighbors thus creates a hypergraph between $\CK_3$ and $K_5$. By Lemma \ref{lemma: no free between G3,G4,G5,G6 and K_5} this contradicts the freeness of $\CA_G$. Finally, assume that $\{1,4,6\}$ is an edge of $G$. Contracting the edge $\{1,4\}$ transforms the hyperedge $\{1,4,6\}$ into the edge $\{4,6\}$, forcing $G$ to have $\CK_6$ as a subgraph and Lemma \ref{lemma: no free between G3,G4,G5,G6 and K_5} contradicts the freeness of $\CA_G$.
\end{proof}

\begin{lemma}\label{lemma: G8 subgraph totally non-free}
    Let $G=(N,E)$ be a hypergraph and let $\mathscr{C}_G$ be as in Definition \ref{Definition: graphs created by 2 building sets}. If one of the $\underline{G}_i$ has $\CK_8$ as a subgraph, then $\CA_G$ is totally non-free.
\end{lemma}

\begin{proof}
    By Lemma \ref{lemma: Induced subgraph is localization}, it is sufficient to choose $N_1\subseteq N$ such that $G[N_1]$ has $\CK_8$ as a subgraph and prove the statement for $G=G[N_1]$. 
    Let $X=(H_{1234}\cap H_{1236}\cap H_{1245}\cap H_{1267}\cap H_{1456})\in L(\CA_{\CK_8})$, then we get the localization
    $(\CA_{\CK_8})_X=\{H_{1234}, H_{1236}, H_{1245}, H_{1267}, H_{1456}, H_{1467}\},$ which is a generic arrangement of rank $5$. The localization of $\CA_{K_7}$ at $X$ is $$(\CA_{K_7})_X=\{H_{1234}, H_{1237}, H_{1236}, H_{1245}, H_{1256}, H_{1267}, H_{1346}, H_{1456}, H_{1467}\},$$
    and thus, $(\CA_{K_7})_X\setminus(\CA_{\CK_8})_X=\{H_{1237}, H_{1256}, H_{1346}\}$.
    So if $\CA_G$ is free, then at least one of the elements of $(\CA_{K_7})_X\setminus(\CA_{\CK_8})_X$ must be included in $\CA_G$ (otherwise $\CA_X$ contradicts the freeness of $\CA_G$). First, assume that $H_{1237}\in\CA_G$ and define $X_1=(H_{1234}\cap H_{1237}\cap H_{1246})$. Then $$(\CA_G)_{X_1}=(\CA_{K_7})_{X_1}=\{H_{1234}, H_{1237}, H_{1246}, H_{1267}\},$$  which is a generic localization of rank $3$. If $H_{1256}\in\CA_G$, we permute the indices of all hyperplanes in $(\CA_G)_{X_1}$ through the permutation $(2 6)(3 4 5)$ and if $H_{1346}\in\CA_G$ we permute the indices of all hyperplanes in $(\CA_G)_{X_1}$ through the permutation $(2 6 7)$. We end up at the generic localizations (both of which have exactly three hyperplanes that are already part of $\CA_{\CK_8}$) $$\{H_{1256}, H_{1267}, H_{1456}, H_{1467}\}\text{ or }\{H_{1236}, H_{1267}, H_{1346}, H_{1467}\}, \text{ respectively}.$$ These are generic localizations of rank $3$ in $\CA_{K_7}$. We conclude that there cannot be any connected hypersubgraph arrangements without a generic localization between $\CA_{\CK_8}$ and $\CA_{K_7}$.
\end{proof}

For the following result we use the same notation as in Definition \ref{Definition: graphs created by 2 building sets}.

\begin{proposition}\label{Proposition: underlying connected subgraph arrangment is free}
    Let $G=(N,E)$ be a hypergraph such that the connected hypersubgraph arrangement $\CA_G$ is free, and let $\CA_{\underline{G}} = \CA_{\underline{G}_1}\times\CA_{\underline{G}_2}\times\dots\times\CA_{\underline{G}_k}$ be the underlying connected subgraph arrangement. Then every $\underline{G}_i\in\mathscr{C}_G$ is either a path-graph, an almost-path-graph, a path-with-triangle-graph, or a cycle-graph.\\
\end{proposition}
\begin{proof}
    We proceed analogously to \cite[Theorem~6.7]{cuntzkuehne:subgrapharrangements}. Let $G=(N,E)$ be a hypergraph such that $\CA_G$ is free and let $\underline{G}_i\in\mathscr{C}_G$ be arbitrary.\\
    
    \noindent 1. If $\underline{G}_i$ has a vertex with four neighbors, it has $\CK_3$ as a subgraph and cannot be free because of Lemma \ref{lemma: no free between G3,G4,G5,G6 and K_5}, which contradicts our assumption.\\
    
    \noindent 2. Assume that $\underline{G}_i$ has two cycles $C_1,C_2$ of at least length $3$ that share at least one edge $\{i_1,i_2\}$. Let $i_3$ be a vertex of $C_1$ but not $C_2$, and let $i_4$ be a vertex of $C_2$ but not $C_1$ (such vertices exist since the length of the cycles is greater than 3 and the cycles are distinct). Now contract $\CA_G$ along all edges of $\underline{G}_i$ (i.e., sets with two elements of $G$)  until only edges containing exactly two elements of $\{i_1,i_2,i_3,i_4\}$ remain. This graph is equal to $\CK_1$ or $\CK_2$, and we use Lemma \ref{lemma: G1G2 subgraph totally non-free} to derive that $\CA_G$ is not free, which would contradict our assumption.\\
    
    \noindent 3. Assume that $\underline{G}_i$ has two cycles that do not share an edge. Contracting all edges until the cycles share a vertex, the resulting graph has $\CK_3$ as a subgraph again and cannot be free by Lemma \ref{lemma: no free between G3,G4,G5,G6 and K_5}.\\
    
    \noindent 4. Assume that $\underline{G}_i$ has exactly one cycle of length at least $4$. Then $\underline{G}_i$ is either a cycle-graph (and $\CA_{\underline{G}_i}$ is free because of Theorem \ref{theorem: FreeConnectedSubgraphArrangements}) or has $\CK_6$ as a subgraph after contracting some edges and fails to be free by Lemma \ref{lemma: no free between G3,G4,G5,G6 and K_5}.\\
    
    \noindent 5. Let $\{i_1,i_2,i_3\}$ be the unique cycle in $\underline{G}_i$. If the degree of all three vertices is $3$, then $\underline{G}_i$ has $\CK_7$ as a subgraph and cannot be free because of Lemma \ref{lemma: G7 subgraph totally non-free}. So at least one of the three vertices (let us say $i_3$) has degree $2$. If $\underline{G}_i$ had a vertex $i_4$ of degree $3$ different from $i_1$ and $i_2$, one could obtain a graph with a vertex of degree at least $4$ by contracting all edges between $i_1$ and $i_4$
    or $i_2$ and $i_4$. This again contradicts Lemma \ref{lemma: no free between G3,G4,G5,G6 and K_5}.  If $\underline{G}_i$ has another vertex $i_4$ of degree $3$, then contracting edges will give us $\CK_3$ as a subgraph, and therefore contradict the freeness of $\CA_G$ by Lemma \ref{lemma: no free between G3,G4,G5,G6 and K_5}. This implies that every vertex different from $i_1$ and $i_2$ has degree $2$, which forces $\underline{G}_i$ to be a path-with-triangle-graph.\\
    
    \noindent 6. Now let $\underline{G}_i$ be a tree. If $\underline{G}_i$ has two vertices of degree $3$, contracting all edges between them will result in an obstruction to the freeness of $\CA_G$ through Lemma $\ref{lemma: no free between G3,G4,G5,G6 and K_5}$.\\
    
    \noindent 7. If $\underline{G}_i$ is not a path-graph, we consider the unique vertex $i_1$ of $\underline{G}_i$ of degree $3$. If all of its neighbors have degree $2$, then $\underline{G}_i$ has $\CK_8$ as a subgraph, which contradicts the freeness of $\CA_G$ by Lemma \ref{lemma: G8 subgraph totally non-free}. So $\underline{G}_i$ is either a path-graph or an almost-path-graph, meaning it has to be equal to one of the graphs of Theorem \ref{theorem: FreeConnectedSubgraphArrangements}, which finishes our proof.
    \end{proof}

Combining Proposition \ref{Proposition: underlying connected subgraph arrangment is free} with Theorem \ref{theorem: FreeConnectedSubgraphArrangements} we derive Theorem \ref{theorem: A_B2 is free}.

\begin{example}\label{ex:freeness_condition}
The converse statement is not true. In Figure \ref{fig:ex_4_nonfree_B2free}, there are three connected hypersubgraph arrangements whose underlying connected subgraph arrangement is free, yet they themselves are not. 
\begin{figure}
\centering
		\includegraphics[width=0.5\textwidth]{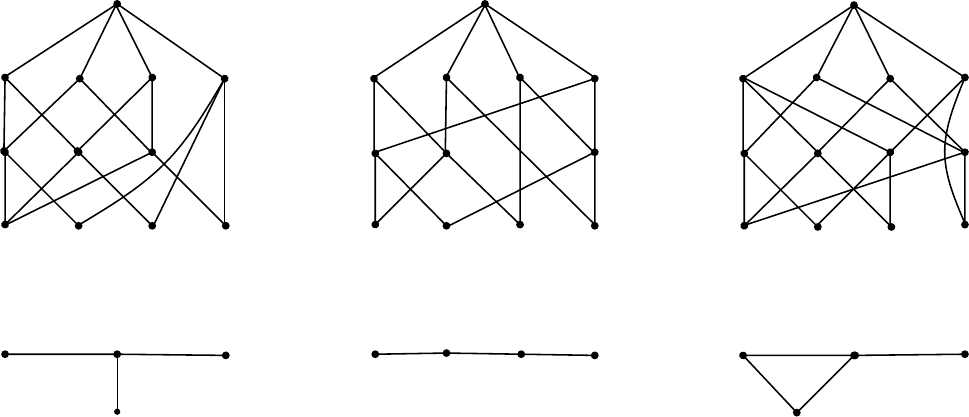} 
		\caption{The posets of the edges of two graphs and their underlying $B_2$-graph. The corresponding arrangements are not free.} 
		\label{fig:ex_4_nonfree_B2free}
\end{figure}
\end{example}

\subsection{Adding additional hyperplanes to free connected subgraph arrangements}
Let $G=(N,E)$ be a hypergraph and assume that $\CA_G$ is free. We know by Theorem \ref{theorem: A_B2 is free} that the underlying connected subgraph arrangement has to be free. So if $\CA_{\underline{G}}= \CA_{\underline{G}_1}\times\CA_{\underline{G}_2}\times\dots\times\CA_{\underline{G}_k}$ is the underlying connected subgraph arrangement with $\underline{G}_i=(N_i,\underline{E}_i)$, then $\CA_{G[N_i]}$ is a connected hypersubgraph arrangement between $\CA_{\underline{G}_i}$ and $\CA_{K_n}$. In this section we will investigate which restrictions we can impose on the hyperedges being contained in $\CA_{G[N_i]}$, if we want $\CA_G$ to be free. 

We look at a specific type of hypergraph that will become a key localization in our proofs. 

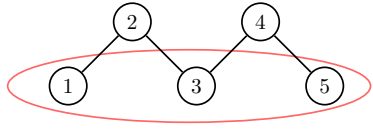
\begin{figure}[H]
\scalebox{0.8}{
\begin{tikzpicture}[node distance={15mm}, thick, main/.style = {draw, circle}] 
\filldraw[ color=red!60, fill=white!1, thick](2,0) ellipse (3 and 0.6);
\node[main] (1) {$1$}; 
\node[main] (2) [above right of=1] {$2$}; 
\node[main] (3) [below right of=2] {$3$}; 
\node[main] (4) [above right of=3] {$4$}; 
\node[main] (5) [below right of=4] {$5$}; 
\draw[-] (1) to (2); 
\draw[-] (2) to (3); 
\draw[-] (3) to (4); 
\draw[-] (4) to (5); 
\end{tikzpicture}}
\caption{The path-graph $P_5$ with one hyperedge added. The corresponding connected
	hypersubgraph arrangement is totally non-free.}\label{fig:graph on 5}
\end{figure}

\begin{lemma}\label{lemma: three connected components generic loca}
    Let $\CA=\{H_{123},H_{234},H_{345},H_{135}\}$, then $\CA$ is a generic arrangement and a localization of rank $3$ in $\CA_{K_5}$. If $\CA_G$ is a connected hypersubgraph arrangement with $\CA$ as a localization, then $\CA_G$ is totally non-free. In particular, if $G$ is a graph with the graph from Figure \ref{fig:graph on 5} as a subgraph, then $\CA_G$ is totally non-free.
\end{lemma}
\begin{proof}
    We have $-\alpha_{H_{123}}+\alpha_{H_{135}}+\alpha_{H_{234}}=\alpha_{H_{345}}$, which shows that for $X=H_{123}\cap H_{135}\cap H_{234}$, we have $\mathcal{A}\subseteq(\CA_{K_5})_X$. As no other $\alpha_{H_I}$ for $H_I\in\CA_{K_5}$ lies in the span of these linear forms, we conclude $\mathcal{A}=(\CA_{K_5})_X$. It is left to show that $\CA$ is generic. Since for all choices of $H_I,H_J\in\CA$ we have $I\cap J\neq \emptyset$ and $I\subsetneq J$, it follows from Lemma \ref{lemma: connected subgraph arrangements are locallyA2} for $X=H_I\cap H_J\in L(\CA_{K_5})$ that $\vert\CA_X\vert=2$. This shows that $\CA$ is a generic arrangement.
     For the second part, if $G$ has the graph from Figure \ref{fig:graph on 5} as subgraph, then we can contract $G$ to this graph and immediatly recover $\CA$ as a subarrangement, and thus as a localization of $\CA_G$. 
\end{proof}
\begin{remark}
The proofs in this section are based on the following idea: We start at some connected hypergraph $G$ with underlying graph $\underline{G}$ and contract hyperedges until we either transform $\underline{G}$ in a way that creates one of the forbidden graphs (see Theorem \ref{theorem: A_B2 is free}) or we find a non-free localization like the one in Lemma \ref{lemma: three connected components generic loca}. Lemmas \ref{lemma: Induced subgraph is localization} and \ref{lemma: contraction of building set} show that the non-free arrangement is a localization of our original arrangement, which cannot be free by Theorem \ref{theorem: localizations are free}.
\end{remark}

\begin{lemma}\label{Lemma: three connected components totally non-free PnCn}
    Let $\underline{G}=(N,E)$ be equal to $P_n$ or $C_n$, where $n\geq 5$. Let $e\subsetneq N$ with $H_e\not\in\CA_{\underline{G}}$, and let $\CA$ be a connected hypersubgraph arrangement between $\CA_{\underline{G}}$ and $\CA_{K_n}$ with $H_e\in\CA$. \\
    If $\underline{G}[e]$ has at least three connected components, then $\CA$ is totally non-free. 
\end{lemma}
\begin{proof}
    Choose an edge $e$ such that $\underline{G}[e]$ has at least three connected components. Since $\underline{G}$ is connected, we can contract edges until $\underline{G}[e]$ has exactly three connected components. Proceed to contract every connected component of $e$ and every connected component of $\underline{G}[N\setminus e]$ to one vertex. The graph created through these operations is still connected and therefore has the hypergraph from Figure \ref{fig:graph on 5} as a subgraph or is equal to it,  and we use Lemma \ref{lemma: three connected components generic loca} to conclude that $\CA$ is totally non-free. \qedhere
    
\end{proof}

We give an example of how we may use this lemma together with the tools of Section \ref{sec:operations} to determine if an arrangement is free. 
\begin{example}\label{example: Hypergraph contraction example}
    Define the hypergraph $G=([9],\{\{1,4,5,6,8,9\},\{i,i+1\}, 1\leq i\leq 8\})$., which is depicted in Figure \ref{example: example generisch kapitel 5}, where we color the hyperedge $e=\{1,4,5,6,8,9\}$ in red.

\begin{figure}
\centering
\scalebox{0.8}{
\begin{tikzpicture}[node distance={15mm}, thick, main/.style = {draw, circle}] 
\filldraw[color=red!60, fill=white!1, thick](5,-1.1) ellipse (5 and 0.7);
\node[main] (2) {$2$}; 
\node[main] (3) [right of=2] {$3$}; 
\node[main] (1) [below right  of=2] {$1$}; 
\node[main] (4) [right of=1] {$4$}; 
\node[main] (5) [right of=4] {$5$};
\node[main] (6) [right of=5] {$6$};
\node[main] (7) [above right of=6] {$7$};
\node[main] (8) [below right of=7] {$8$};
\node[main] (9) [right of=8] {$9$};
\draw[-] (1) to (2); 
\draw[-] (2) to (3); 
\draw[-] (3) to (4); 
\draw[-] (4) to (5);
\draw[-] (5) to (6);
\draw[-] (6) to (7);
\draw[-] (7) to (8);
\draw[-] (8) to (9);
\end{tikzpicture}} 
\caption{Sketch of the hypergraph of Example \ref{example: Hypergraph contraction example}. The corresponding connected hypersubgraph
	arrangement is totally non-free. Note that one can generate the hypergraph from
	Figure 5 through repeated edge contractions.}\label{example: example generisch kapitel 5}
\end{figure}
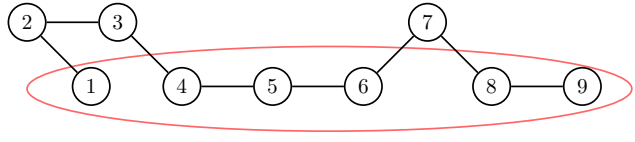

Note that $G$ is just $P_9$ with an additional hyperedge, so the connected hypersubgraph arrangement $\CA_G$ lies between $\CA_{P_9}$ and $\CA_{K_9}$. The graph $P_9[e]$ has three different connected components ($\{1\},\{4,5,6\},\{8,9\}$) and if we delete the vertices contained in $e$ from $P_9$, then two connected components ($\{2,3\},\{7\}$) remain. We contract all of the five mentioned connected components to one vertex and end up at the graph from Figure \ref{fig:graph on 5}. By Lemma \ref{lemma: contraction of building set}, the corresponding connected hypersubgraph arrangement is a localization of $\CA_G$, however, due to Lemma \ref{Lemma: three connected components totally non-free PnCn}, this arrangement is totally non-free, so $\CA_G$ is totally non-free by Lemma \ref{lemma: contraction of building set}.
\end{example}

This shows that if, for a simple graph $\underline{G}\in\{P_5, C_5\}$ underlying a hypergraph $G$, there exists a hyperedge $e\in G$ such that $\underline{G}[e]$ has three connected components, then $\CA_G$ cannot be free. We extend this observation to the cases where $\underline{G}\in\{A_{k,n},\Delta_{k,n}\}$.  

\begin{lemma}\label{Lemma: three connected components totally non-free AknDkn}
    Let $\underline{G}\in\{A_{k,n},\Delta_{k,n}\}$ with vertex set $N=\{1,2,\dots,n+1\}$ (where $n\geq 4$) and edge set $E$. Let $e\subsetneq N$ with $H_e\not\in\CA_{\underline{G}}$, and $\CA$ be a connected hypersubgraph arrangement between $\CA_{\underline{G}}$ and $\CA_{K_n}$ with $H_e\in\CA$. Assume that $\underline{G}[e]$ has at least three connected components. 
    \begin{enumerate}
        \item If $n+1\not\in e$, then $\CA$ is totally non-free.
        \item If $n+1\in e$, then $\CA$ is not free.
    \end{enumerate}
\end{lemma}
\begin{proof}
    1. Choose an edge $e$ such that $\underline{G}[e]$ has at least three connected components. If $n+1\not\in e$, then $e$ is contained in $\underline{G}[[n]]\simeq P_n$, and we are finished using Lemmas \ref{lemma: Induced subgraph is localization} and \ref{Lemma: three connected components totally non-free PnCn}.\\ 
    \noindent 2. Contracting connected components as before, we either end up at a path-graph with additional hyperedges as in Lemma \ref{Lemma: three connected components totally non-free PnCn} and are finished, or we end up at a hypergraph containing the hypergraph depicted in Figure \ref{figure: three connected figure} as a subgraph.
\begin{figure}
\centering
\scalebox{0.8}{
\begin{tikzpicture}[node distance={15mm}, thick, main/.style = {draw, circle}] 
\filldraw[color=red!60, fill=white!1, thick](1.5,0) ellipse (2 and 0.7);
\node[main] (2) {$1$};  
\node[main] (3) [right of=2] {$3$}; 
\node[main] (1) [above of=3] {$2$}; 
\node[main] (4) [right of=3] {$4$}; 
\draw[-] (1) to (2); 
\draw[-] (3) to (1); 
\draw[-] (4) to (1); 
\end{tikzpicture}}
\caption{A hypergraph as in Lemma \ref{Lemma: three connected components totally non-free AknDkn}. The corresponding connected hypersubgraph arrangement is totally non-free.}\label{figure: three connected figure}
\end{figure}
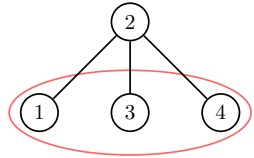
\noindent Note that the connected hypersubgraph arrangement of this hypergraph is $\CA_{G_2}\setminus\{H_{23},H_{24},H_{34}\}$. If none of these three hyperplanes is included (as a result of the previous contractions), then the localization $\{H_{2},H_{3},H_{4},H_{234}\}$ is a generic arrangement, and we use Theorem \ref{theorem: Yoshinaga generic not free}. If any two or even all three of these hyperplanes are included in the arrangement, the graph is equal to $\CK_1$ or $\CK_2$, and the connected hypersubgraph arrangement is not free by Proposition \ref{Proposition: underlying connected subgraph arrangment is free}. So assume that only $H_{23}$ is missing, and let $X=(H_1\cap H_4\cap H_{123})$, then we have $(\CA_{\CK_2})_X=\{H_1, H_4, H_{14}, H_{123}, H_{234}, H_{1234}\}$. This localization fails to be free, since it is isomorphic to the arrangement consisting of the hyperplanes $\{H_1,H_2,H_3,H_{12},H_{13},H_{23}\}$. All free multiplicities on this so-called $X_3$-arrangement were first classified by DiPasquale and Wakefield in \cite{dipasquale:modulifreeness}. An isomorphism is given by the linear map defined through $$x_1\mapsto x_1, x_2\mapsto x_4, x_3\mapsto -(x_1+x_2+x_3+x_4).$$ The localizations in case $\CA_{\CK_2}$ only contains $H_{24}$ or $H_{34}$ are retrieved by permutation of the indices of the hyperplanes. To be precise, if only $H_{24}$ is contained, then $(\CA_{\CK_2})_X=\{H_1, H_3, H_{13}, H_{124}, H_{234}, H_{1234}\}$, and if only $H_{34}$ is contained, then $(\CA_{\CK_2})_X=\{H_1, H_2, H_{12}, H_{134}, H_{234}, H_{1234}\}$ fails to be free.\\
This shows that, with any number of additional edges, the corresponding connected hypersubgraph arrangement cannot be free, which concludes our proof.
\end{proof}

\noindent Because of Lemmas \ref{Lemma: three connected components totally non-free PnCn} and \ref{Lemma: three connected components totally non-free AknDkn} it remains to investigate possible edges $e\subseteq N$, where $\underline{G}[e]$ has exactly two connected components (if it only had one connected component, then already $H_e\in\CA_{\underline{G}}$ by definition). However, for $\underline{G}=C_n$ there do not exist any free connected hypersubgraph arrangements between $\CA_{\underline{G}}$ and $\CA_{K_n}$.

\begin{proposition}\label{proposition: no free between Cn and Kn}
    Let $\CA$ be a connected hypersubgraph arrangement between $\CA_{C_n}$ and $\CA_{K_n}$, then $\CA$ is not free.
\end{proposition}
\begin{proof}
    If $\underline{G}[e]$ has three connected components, we use Lemma \ref{Lemma: three connected components totally non-free PnCn} to show that $\CA$ is totally non-free. If $\underline{G}[e]$ has one connected component, then $H_e\in\CA_{\underline{G}}\subseteq \CA$ already. So assume that $\underline{G}[e]$ has exactly two connected components. Contracting the connected components of $\underline{G}[e]$ and $\underline{G}[N\setminus e]$ to one vertex, we end up at one the hypergraphs depicted in Figure \ref{fig:5.13}, where $e$ is colored in red. 
\begin{figure}\label{fig:enter-label}
    \centering
    \scalebox{0.8}{
\begin{tikzpicture}[node distance={15mm}, thick, main/.style = {draw, circle}]
\filldraw[color=red!60, fill=white!1, thick](0,-1.1) ellipse (0.5 and 1.75);
\node[main] (1) {$1$}; 
\node[main] (2) [below right of=1] {$2$}; 
\node[main] (3) [below left of=2] {$3$}; 
\node[main] (4) [above left of=3] {$4$}; 
\draw[-] (1) to (2); 
\draw[-] (2) to (3); 
\draw[-] (3) to (4); 
\draw[-] (4) to (1);
\end{tikzpicture}
\hspace{3cm}
\begin{tikzpicture}[node distance={15mm}, thick, main/.style = {draw, circle}] 
\filldraw[color=red!60, fill=white!1, thick](0,-1.1) ellipse (0.5 and 1.75);
\node[main] (1) {$1$}; 
\node[main] (2) [below right of=1] {$2$}; 
\node[main] (3) [below left of=2] {$3$}; 
\node[main] (4) [above left of=3] {$4$}; 
\draw[-] (1) to (2); 
\draw[-] (2) to (3); 
\draw[-] (3) to (4); 
\draw[-] (4) to (1);
\draw[-] (4) to (2);
\end{tikzpicture}}
\caption{Hypergraphs whose corresponding connected hypergraph arrangement is not
	free. They are required for the proof of Proposition \ref{proposition: no free between Cn and Kn}.}
\label{fig:5.13} 
\end{figure}
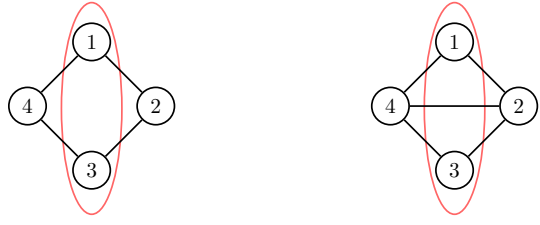

Since the corresponding connected hypersubgraph arrangement of these graphs is equal to $\CA_{\CK_1}$ or $\CA_{\CK_2}$, it is not free due to Theorem \ref{theorem: FreeConnectedSubgraphArrangements}.
\end{proof}

\subsection{Constructing free connected hypersubgraph arrangements}
Given two free, irreducible $(0/1)$-arrangements, we can create a new canonical free multiarrangement in the following way.

\begin{theorem}\label{theorem: glueing irreducible free arrangements together}
Assume $(\CA_1,V_1), (\CA_2,V_2)$ are free $(0/1)$-arrangements with $\rank \CA_1=r_1,\rank\CA_2=r_2$, and $H_{1\dots r_1}\in\CA_1, H_{r_1+1\dots r_1+r_2}\in\CA_2$. Then $\CA :=(\CA_1\times\CA_2\cup\{H_{1\dots r_1+r_2}\}, V_1\oplus V_2)$ is free.
\end{theorem}
\begin{proof}
 Both $\CA_1,\CA_2$ are free, so for a fixed $i\in\{1,2\}$ the module $D(\CA_i)$ is free, and we fix bases $B_1=\{\mu_E,\mu_2, \dots,\mu_{r_1}\}$ and $B_2 = \{\nu_E, \nu_2, \dots, \nu_{r_2}\}$, where $\mu_E$ and $\nu_E$ are the respective Euler derivations.
 We use the canonical inclusion of $D(\CA_i)\to D(\CA_1\times\CA_2)$ by defining $S_{V_1} = \mathbb{K}[x_1,\dots, x_{r_1}]$ and $S_{V_2} = \mathbb{K}[x_{r_1+1},\dots, x_{r_1+r_2}]$, such that $S_{V_1\oplus V_2} = \mathbb{K}[x_1, \dots, x_{r_1+r_2}]$, and we consider $\text{Der}_{S_1\oplus S_2}$.\\

Adding $H_{\max}:=H_{1\dots r_1+r_2}$ to $\CA_1\times\CA_2$, we construct a new basis based on $B_1$ and $B_2$. 
We choose the Euler Derivation $\theta_E= \mu_E + \nu_E = \sum_{i=1}^{r_1+r_2} x_i \partial_{x_i}\in D(\CA)$ as a basis element of degree $1$ and $\theta_N=\alpha_{H_{\max}} \cdot \nu_E\in D(\CA)$ as basis element of degree $2$. 

\medskip 
It holds that $\restrtwo{\mu}{S_{V_2}} = 0$ and $\restrtwo{\nu}{S_{V_1}} = 0$ for all $\mu \in \text{Der}(\A_1), \nu \in \text{Der}(\A_2)$. 
For $\mu_i \in B_1, \nu_i \in B_2, 2 \leq i$, we define new derivations as follows: \[\theta_i^{\mu} = \mu_i + \frac{\mu_i(\alpha_{H_{1\dots r_1}})}{\alpha_{H_{1\dots r_1}}}\cdot \nu_E~\text{and}~\theta_i^{\nu} = \nu_i + \frac{\nu_i(\alpha_{H_{r_1\dots r_1+r_2}})}{\alpha_{H_{r_1\dots r_1+r_2}}}\cdot \mu_E\]

It is clear that the condition for these derivations $\theta_i^\mu, \theta_i^{\nu}$ to be elements of $\text{Der}(\A)$ is met for all $H \in \A_1 \times \A_2$. For $H_{\max}$, we have 
\begin{align*}
   \theta^{\mu}_i(\alpha_{H_{\max}}) &= \theta^{\mu}_i(\alpha_{H_{1 \dots r_1}})+\theta^{\mu}_i(\alpha_{H_{r_1+1\dots,r_1+r_2}})\\
    &= \frac{\mu_i(\alpha_{H_{1\dots r_1}})}{\alpha_{H_{1\dots r_1}}} \cdot \alpha_{H_{1 \dots r_1}} + \frac{\mu_i(\alpha_{H_{1\dots r_1}})}{\alpha_{H_{1\dots r_1}}} \cdot\alpha_{H_{r_1+1\dots r_1+r_2}}\\
    &=\frac{\mu_i(\alpha_{H_{1\dots r_1}})}{\alpha_{H_{1\dots r_1}}} \cdot\alpha_{H_{\max}}\in \alpha_{H_{\max}}\cdot S_{V_1\oplus V_2}.
\end{align*}

Showing $\theta_i^{\nu}(\alpha_{H_{\max}}) \in \alpha_{H_{\max}}\cdot S_{V_1\oplus V_2}$ is an analogous computation. Computing the determinant of $M(\theta_E, \theta_N, \theta^\mu_2, \dots , \theta^\mu_{r_1}, \theta^\nu_2, \dots, \theta^\nu_{r_1+r_2})$ in order to use Saito's criterion (Theorem \ref{thm:saito}), we perform some row operations: 
\begin{align*}
    \det(M(\theta_E, \theta_N, \theta^\mu_2, \dots , \theta^\mu_{r_1}, \theta^\nu_2, \dots, \theta^\nu_{r_1+r_2})) &= \alpha_{H_{\max}} \cdot \det(M(\theta_E, \nu_E, \theta^\mu_2, \dots , \theta^\mu_{r_1}, \theta^\nu_2, \dots, \theta^\nu_{r_1+r_2}))\\  &= \alpha_{H_{\max}} \cdot \det(M(\theta_E-\nu_E, \nu_E, \theta^\mu_2, \dots , \theta^\mu_{r_1}, \theta^\nu_2, \dots, \theta^\nu_{r_1+r_2})) \\ &= \alpha_{H_{\max}} \cdot \det(M(\mu_E, \nu_E, \theta^\mu_2, \dots , \theta^\mu_{r_1}, \theta^\nu_2, \dots, \theta^\nu_{r_1+r_2})).
\end{align*}
The entries for the Euler derivations are just the variables $x_i$ themselves; thus, it follows again with basic row operations:
\begin{align*}
    \alpha_{H_{\max}} \cdot \det(M(\mu_E, \nu_E, \theta^\mu_2, \dots, \theta^\nu_{r_1+r_2})) &=  \alpha_{H_{\max}} \cdot \det(M(\mu_E, \nu_E, \mu_2, \dots , \mu_{r_1}, \nu_2, \dots, \nu_{r_1+r_2})) \\ &= \alpha_{H_{\max}} \cdot (-1)^{r_1-1} \cdot \det(M(\mu_E, \mu_2, \dots , \mu_{r_1},\nu_E, \nu_2, \dots, \nu_{r_1+r_2})) \\ &= \alpha_{H_{\max}} \cdot (-1)^{r_1-1} \cdot \det(M(\mu_E, \mu_2, \dots , \mu_{r_1}) \cdot \det(M(\nu_E, \nu_2, \dots, \nu_{r_1+r_2}))\\ &= c \cdot \mathcal{Q}(\A),
\end{align*}
where the last inequality comes from the fact that the matrix is a block matrix and Saito's criterion holds for the bases $B_1$ and $B_2$. Applying Saito's riterion, we can assert that these derivations form a basis and $\A$ is free. 
\end{proof}

We now demonstrate that Theorem \ref{theorem: glueing irreducible free arrangements together} can be utilized to construct infinitely many free connected hypersubgraph arrangements, that are not connected subgraph arrangements.

\begin{remark}
Using the notation as in Theorem \ref{theorem: glueing irreducible free arrangements together}, let $\CA_1,\CA_2$ be connected hypersubgraph arrangements with respective hypergraphs $G_1,G_2$. Then, we have the following:
    \begin{enumerate}
        \item Adding $H_{\max}$ to $\CA_1\times\CA_2$ is just the connected hypersubgraph arrangement stemming from the hypergraph that gets created when connecting $G_1$ and $G_2$ by the hyperedge $V(G_1)\cup V(G_2)$.
        \item If $\CA_1,\CA_2$ in Theorem \ref{theorem: glueing irreducible free arrangements together} are connected hypersubgraph arrangements, then according to Lemma \ref{lemma: G not in B then AB reducible} the condition $H_{1 \dots r_1}\in\CA_1, H_{r_1+1 \dots r_1+r_2}\in\CA_2$ is equivalent to $\CA_1$ and $\CA_2$ being irreducible. Therefore, we can create new free, irreducible connected hypersubgraph arrangements by "gluing them together" by utilizing Theorem \ref{theorem: glueing irreducible free arrangements together}.
        \item Let $\CA_1=\CA_2=\CA_{P_2}$. By Theorem \ref{theorem: glueing irreducible free arrangements together}, the arrangement $(\CA_1\times\CA_2\cup\{H_{1234}\},\mathbb{Q}^4)$ is free with exponents $(1,2,2,2)$. Since there does not exist an irreducible connected subgraph arrangement with seven hyperplanes, the newly created connected hypersubgraph arrangement cannot be a connected subgraph arrangement. This process is visualized in Figure \ref{fig: Glueing P2 and P2}.
        \item We want to mention that there exist connected hypersubgraph arrangements that are neither connected subgraph arrangements, nor constructed through Theorem \ref{theorem: glueing irreducible free arrangements together}. An example is the hypergraph $G=(N,E)$, where $$N=\{1,2,3,4\}\text{ and }E=\{\{1\},\{2\},\{3\},\{4\},\{1, 2\},\{3, 4\},\{1, 2, 3\},\{1, 2, 3, 4\}\}.$$ Then $\CA_G$ is free with exponents $(1,2,2,3)$. However, removing $H=H_{1234}$ from $\CA_G$ yields a free arrangement with exponents $(1,2,2,2)$, and therefore the deletion is not reducible.
    \end{enumerate}
\end{remark}

\begin{figure}
    \centering
    \scalebox{0.8}{
    \begin{minipage}[c]{0.45\textwidth}
        \centering
        \begin{tikzpicture}[node distance={15mm}, thick, main/.style = {draw, circle}] 
            \node[main] (1) {$1$}; 
            \node[main] (2) [right of=1] {$2$}; 
            \node[main] (3) [right of=2] {$3$}; 
            \node[main] (4) [right of=3] {$4$}; 
            \draw[-] (1) to (2); 
            \draw[-] (3) to (4); 
        \end{tikzpicture}
    \end{minipage}
    \hspace{1cm}
    \begin{minipage}[c]{0.45\textwidth}
        \centering
        \begin{tikzpicture}[node distance={15mm}, thick, main/.style = {draw, circle}] 
            \filldraw[color=red!60, fill=white!1, thick](2.25,0) ellipse (3 and 0.75);
            \node[main] (1) {$1$}; 
            \node[main] (2) [right of=1] {$2$}; 
            \node[main] (3) [right of=2] {$3$}; 
            \node[main] (4) [right of=3] {$4$}; 
            \draw[-] (1) to (2); 
            \draw[-] (3) to (4); 
        \end{tikzpicture}
    \end{minipage}}
    \caption{The free and reducible connected subgraph arrangement $\CA_{P_2}\times\CA_{P_2}$ and the irreducible connected hypersubgraph arrangement created through Theorem \ref{theorem: glueing irreducible free arrangements together}.}
    \label{fig: Glueing P2 and P2}
\end{figure}
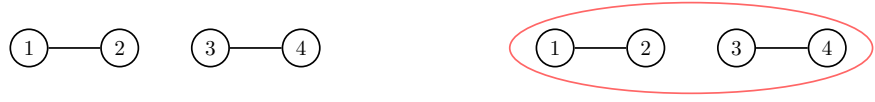

\section{Simplicial arrangements} \label{sec:simplicial}
Cuntz and Kühne classified all connected subgraph arrangements that are simplicial (see Definition \ref{def:simp}). Being simplicial is a local property for arrangements, as proven in \cite[Proposition~3.2]{stanley:supersolvablelocal}. The following is well known and follows immediately from the fact that simplicial arrangements are $K(\pi, 1)$ \cite{deligne:immeubles}, while generic arrangements are not \cite{hattori:topology}.

\begin{lemma}\label{lemma: generic arrangement not simplicial}
    Let $\CA$ be a generic arrangement in $\mathbb{K}^n$ with $n\geq 3$ and $\vert\CA\vert=h\geq n+1$. Then $\CA$ is not simplicial.
\end{lemma}

A very important tool to check if a given arrangement is simplicial is the following result by Cuntz and Geis. 

\begin{proposition}\label{proposition: criteria for simplicity}\cite[Corollary~2.4]{cuntzgeis:simplicialcriteria}
    Let $\CA$ be a central essential arrangement of hyperplanes in $\mathbb{R}^n, n\geq 2$. Then $\CA$ is simplicial if and only if $$n\cdot \vert\mathcal{K}(\CA)\vert=2\sum_{H\in\CA} \vert \mathcal{K}(\CA^H)\vert.$$
\end{proposition}

Due to a result by Zaslavsky, we can calculate the number of chambers $\mathcal{K}(\CA)$ in order to apply Proposition \ref{proposition: criteria for simplicity} by evaluating characteristic polynomials (see Definition \ref{def:charac_pol}).

\begin{theorem}\cite[Zaslavsky's Theorem]{zaslavsky:chamberformula}\label{proposition: zaslavsky simplicial criteria}
   $\vert\mathcal{K}(\CA)\vert=(-1)^n\chi(\CA,-1).$
\end{theorem}

\begin{proposition}\label{proposition: underlying connected subgraph simplicial}
    Let $G=(N,E)$ be a hypergraph. If $\CA_G$ is simplicial, then $\CA_{\underline{G}}$ is simplicial. This means every $\underline{G}_i$ in the decomposition of $\CA_{\underline{G}}=\CA_{\underline{G}_1}\times\dots\times \CA_{\underline{G}_k}$ is either equal to $C_3$ or some $P_n$.
\end{proposition}

\begin{proof}
    \noindent Let $N_i$ be the set of vertices of $\underline{G}_i$. By exchanging $G$ with $G[N_i]$, we can assume that $k=1$ (i.e., $\CA_{\underline{G}}$ is irreducible). Cuntz and Kühne proved \cite[Theorem~7.2]{cuntzkuehne:subgrapharrangements} that a connected subgraph arrangement $\CA_{\underline{G}}$ is simplicial if and only if the graph $\underline{G}$ is a triangle ($C_3$) or a path-graph. They showed that if a graph $\underline{G}$ has a vertex with at least $3$ neighbors, then the induced subgraph of $\underline{G}$ on those four vertices is not simplicial. Therefore $\CA_{\underline{G}}$ cannot be simplicial due to Lemma \ref{lemma: Induced subgraph is localization}. As a first step we show that there exists no hypergraph $G$ between $\underline{G}=([4],\{\{1,i\}\mid i=2,3,4\})$ and $K_4$ , such that $\CA_G$ is simplicial. Note that the only genuine hyperedge that adds a new hyperplane to $\CA_{\underline{G}}$ is $\{2,3,4\}$, since all other induced subgraphs of $\underline{G}$ are already connected. So, investigating the connected hypersubgraph arrangements of the hypergraphs in Figure \ref{figure: hypergraphs for prop 6.5} is sufficient, since adding any other edges creates an arrangement that is the connected subgraph arrangement of $G_1$ in Figure \ref{Figure: G1 to G8} or $K_4$ and therefore not simplicial by the result of Cuntz and Kühne.
\begin{figure}
\centering
\scalebox{0.8}{
\begin{tikzpicture}[node distance={10mm}, main/.style = {draw, circle}] 
\filldraw[color=red!60, fill=white!1, thick](0,-1.0) ellipse (1.5 and 0.6);
\node[main] (1) {$1$}; 
\node[main] (3) [below of=1] {$3$}; 
\node[main] (2) [left of=3] {$2$}; 
\node[main] (4) [right of=3] {$4$};
\draw[-] (1) to (2); 
\draw[-] (1) to (3); 
\draw[-] (1) to (4); 
\end{tikzpicture}
\begin{tikzpicture}[node distance={10mm}, main/.style = {draw, circle}] 
\filldraw[color=red!60, fill=white!1, thick](0,-1.0) ellipse (1.5 and 0.6);
\node[main] (1) {$1$}; 
\node[main] (3) [below of=1] {$3$}; 
\node[main] (2) [left of=3] {$2$}; 
\node[main] (4) [right of=3] {$4$};
\draw[-] (1) to (2); 
\draw[-] (1) to (3); 
\draw[-] (1) to (4); 
\draw[-] (2) to (3); 
\end{tikzpicture}}
\caption{Hypergraphs on $4$ vertices, where the corresponding connected hypergraph arrangement
	is not simplicial. They serve as minimal counterexamples in the proof of Proposition \ref{proposition: underlying connected subgraph simplicial}.}\label{figure: hypergraphs for prop 6.5}
\end{figure}
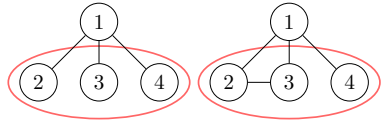
 For both of the pictured hypergraphs, we use Proposition \ref{proposition: criteria for simplicity} and derive the inequalities $880\neq 944$ in the first and $1040\neq 1088$ in the second case, so the connected hypersubgraph arrangements fail to be simplicial. Hence, all vertices of $\underline{G}$ have at most $2$ neighbors.\\ 
 Since all induced subgraphs $C_4[e]$ for a genuine hyperedge $e$ are already connected, we solely have to consider adding edges. States differently, every connected subhypergraph arrangement between $\CA_{C_4}$ and $\CA_{K_4}$ is a connected subgraph arrangement. Now assume that $\underline{G}$ has $C_n$ for $n\geq 4$ as a subgraph. Contracting edges of $C_n$ until it becomes $C_4$ and restricting to the corresponding induced hypersubgraph creates one of the connected subgraph arrangements $\CA_{C_4}, \CA_{G_1},$ or $\CA_{K_4}$, none of which is simplicial. Combining these findings with Lemmas \ref{lemma: Induced subgraph is localization} and \ref{lemma: contraction of building set} finishes the proof.
\end{proof}

\begin{lemma}\label{lemma: circle graph no hyperedge if simplicial}
    Let $G=(N,E)$ be a hypergraph. Assume that in the underlying connected subgraph arrangement $\CA_{\underline{G}}=\CA_{\underline{G}_1}\times\dots\times \CA_{\underline{G}_k}$ there exists an $1\leq i\leq k$ such that $\underline{G}_i=(N_i,\underline{E}_i)=C_3$. If there exists a genuine hyperedge $e\in E$ such that $e\neq N_i$ and $e\cap N_i\neq\emptyset$, then $\CA_G$ is not simplicial.\\
\end{lemma}
\begin{proof}
    Relabel the vertices such that $\underline{G} = \underline{G}[\{1,2,3\}] = C_3 = (N_1,\CE_1)$ and assume $e\neq N_1$ (otherwise $H_e\in \CA_{\underline{G}}$). By Lemma \ref{lemma: Induced subgraph is localization} it suffices to show $G = G[N_1\cup\{e\}]$ is not simplicial. By Proposition \ref{proposition: underlying connected subgraph simplicial}, every graph $\underline{G}_i$ in the underlying subgraph $\underline{G}$ is either equal to $C_3$ or some $P_n$.\\
    Subsequently contract all edges $e'\subseteq \{3,4,\dots,n\}$ and contract any new edges which might get created in the process as well. If during this process an edge $\{i,j\}$ with $i\in\{1,2,3\}, j\in e\setminus\{1,2,3\}$ gets created, then the corresponding arrangement is not simplicial by Proposition \ref{proposition: underlying connected subgraph simplicial} and we are done. After finishing this procedure, $G$ does not contain any edges other than the ones of $\underline{G}$. We showcase the corresponding underlying connected subgraph arrangement on the left in Figure \ref{figure: graphs for simplicity lemma}, where $4\leq j\leq n$.

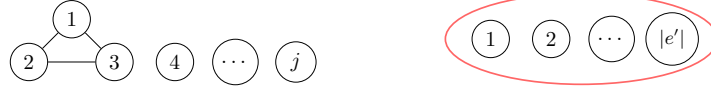
\begin{figure}
\scalebox{0.8}{
\begin{tikzpicture}[node distance={10mm}, main/.style = {draw, circle}] 
\node[main] (1) {$1$}; 
\node[main] (2) [below left of=1] {$2$}; 
\node[main] (3) [below right of=1] {$3$};
\node[main] (4) [right of=3] {$4$};
\node[main] (5) [right of=4] {$\dots$};
\node[main] (6) [right of=5] {$j$};
\draw[-] (1) to (2); 
\draw[-] (3) to (2); 
\draw[-] (1) to (3); 
\end{tikzpicture}\hspace{2cm}
\begin{tikzpicture}[node distance={10mm}, main/.style = {draw, circle}] 
\filldraw[color=red!60, fill=white!1, thick](0.5, 0) ellipse (2.25 and 0.75);
\node[main] (2) {$2$}; 
\node[main] (1) [left of=2] {$1$}; 
\node[main] (3) [right of=2] {$\dots$};
\node[main] (4) [right of=3] {\small{$\vert e'\vert$}};
\end{tikzpicture}}
\caption{First set of hypergraphs, where the corresponding connected hypersubgraph
	arrangement is not simplicial. They are used in the proof of Lemma \ref{lemma: circle graph no hyperedge if simplicial}.}\label{figure: graphs for simplicity lemma}
\end{figure}

\noindent Assume there exists a hyperedge $e'\subseteq \{4,5,\dots,j\}$ with $\vert e'\vert\geq 3$ and choose $e'$ such that no other hyperedge is included in it. Now $G[e']$ is the hypergraph on the right in Figure \ref{figure: graphs for simplicity lemma}, meaning it yields a generic arrangement as in Lemma \ref{lemma: generic arrangement not simplicial}, and $\CA_G$ is not simplicial.

\noindent Now assume that there exists a hyperedge $e$ with $e \cap \{1,2,3\}\neq\emptyset$ and $\vert e\vert\geq 3$ that contains no other hyperedge. This entails that the underlying connected subgraph arrangement of $G[\{1,2,3\}\cup e]$ is contained in Figure \ref{figure: more graphs for simplicity lemma}, where the hyperedge $e$ is shown in red as well.

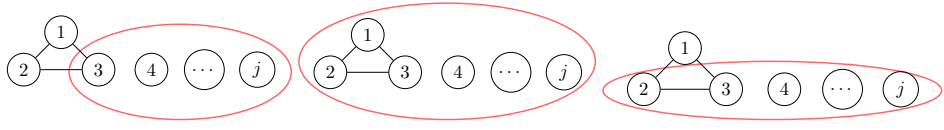
\begin{figure}
\centering
\scalebox{0.7}{
\begin{tikzpicture}[node distance={10mm}, main/.style = {draw, circle}] 
\filldraw[color=red!60, fill=white!1, thick](2.25, -0.75) ellipse (2.1 and 0.9);
\node[main] (1) {$1$}; 
\node[main] (2) [below left of=1] {$2$}; 
\node[main] (3) [below right of=1] {$3$};
\node[main] (4) [right of=3] {$4$};
\node[main] (5) [right of=4] {$\dots$};
\node[main] (6) [right of=5] {$j$};
\draw[-] (1) to (2); 
\draw[-] (3) to (2); 
\draw[-] (1) to (3); 
\end{tikzpicture}
\begin{tikzpicture}[node distance={10mm}, main/.style = {draw, circle}] 
\filldraw[color=red!60, fill=white!1, thick](1.5, -0.5) ellipse (2.8 and 1.1);
\node[main] (1) {$1$}; 
\node[main] (2) [below left of=1] {$2$}; 
\node[main] (3) [below right of=1] {$3$};
\node[main] (4) [right of=3] {$4$};
\node[main] (5) [right of=4] {$\dots$};
\node[main] (6) [right of=5] {$j$};
\draw[-] (1) to (2); 
\draw[-] (3) to (2); 
\draw[-] (1) to (3); 
\end{tikzpicture}
\begin{tikzpicture}[node distance={11mm}, main/.style = {draw, circle}] 
\filldraw[color=red!60, fill=white!1, thick](1.65, -0.8) ellipse (3.2 and 0.55);
\node[main] (3) {$1$};
\node[main] (2) [below right of=3]{$3$}; 
\node[main] (1) [below left of=3] {$2$}; 
\node[main] (4) [right of=2] {$4$};
\node[main] (5) [right of=4] {$\dots$};
\node[main] (6) [right of=5] {$j$};
\draw[-] (1) to (2); 
\draw[-] (1) to (3); 
\draw[-] (3) to (2); 
\end{tikzpicture}} 
\caption{Second set of hypergraphs, where the corresponding connected hypersubgraph
	arrangement is not simplicial. They are used in the proof of Lemma \ref{lemma: circle graph no hyperedge if simplicial}.}\label{figure: more graphs for simplicity lemma}
\end{figure}
Note that in the hypergraph on the top left, we have $j\geq 5$, otherwise $e$ is an edge, and we get a contradiction through Proposition \ref{proposition: underlying connected subgraph simplicial}. But since $e$ does not contain any other hyperedges, $\CA_{G[e]}$ is the generic arrangement from Lemma \ref{lemma: generic arrangement not simplicial} and therefore not simplicial. We can assume that in the second picture, we have $j=4$ and in the third picture, we have $j=3$, because otherwise we contract $C_3$ or $P_2$ to a single vertex, and once again get a contradiction through Lemma \ref{lemma: generic arrangement not simplicial}. Assuming that $j=4$ in the second and $j=3$ in the third picture, we use Theorem \ref{proposition: zaslavsky simplicial criteria} to show that the corresponding connected hypersubgraph arrangements are not simplicial as we end up at $384\neq 404$ and $54\neq 56$, respectively.
\end{proof}

\noindent We are ready to classify all hypergraphs $G$ such that the corresponding connected hypersubgraph arrangement $\CA_G$ is simplicial.

\begin{proof}[Proof of Theorem \ref{thm:simpliciality}]
    Assume that $\CA_G$ is simplicial. We show that there does not exist a genuine hyperedge $e\in E$, and therefore $\CA_G$ is a connected subgraph arrangement. Use Proposition \ref{proposition: underlying connected subgraph simplicial} to deduce that the underlying connected subgraph arrangement is of the form $\CA_{\underline{G}}=\CA_{\underline{G}_1}\times\dots\times \CA_{\underline{G}_k}$ where every $\underline{G}_i=(N_i,\underline{E}_i)$ is equal to a path-graph or $C_3$. Applying Lemma \ref{lemma: circle graph no hyperedge if simplicial}, we see that there does not exist a vertex $j\in[n]$ and a genuine hyperedge $e\in E$ such that $\underline{G}_j=(N_j,\underline{E}_j)=C_3, e\cap N_j\neq \emptyset,$ and $H_e\not\in\CA_{\underline{G}}$. We can conclude that $\underline{G}_i$ is not connected to another connected component of $\underline{G}$, and it is sufficient to investigate the induced subgraph on the set $\cup_{i, (\underline{G}_i\text{ is a path-graph)}}N_i$. Consequently, let every $\underline{G}_i$ be equal to a path-graph. Choose a genuine hyperedge $e\in E$ that does not contain another hyperedge. After possibly having to contract some edges, the arrangement $\CA_{G[e]}$ is either equal to a generic arrangement described in Lemma $\ref{lemma: generic arrangement not simplicial}$, and therefore not simplicial, or to the connected hypersubgraph arrangement coming from the hypergraph defined by $([3],\{\{1,2\},\{1,2,3\}\})$, and not simplicial, as seen in the proof of Lemma \ref{lemma: circle graph no hyperedge if simplicial} which would contradict our assumption that $\CA_G$ is simplicial.\\
    This shows that no genuine hyperedges exist, so $G$ is a graph and $\CA_G$ is a connected subgraph arrangement.
\end{proof}

\section{Further research}
There are a few natural questions that arise from our observations and results thus far: 
\begin{enumerate}
    \item The concept of building sets is more general than just the Boolean lattice case. One might define general building set arrangements analogously to Definition \ref{definition: connected hypersubgraph arrangements} and analyse these arrangements along the lines proposed here (freeness, simpliciality, etc.) or in the context of general building set theory.
    \item Is there a hypergraphic complete characterization of free building set arrangements? 
    \item The approach with the underlying connected subgraph arrangement can be seen as focusing on $E_2 := \{e \in G \mid \vert e \vert = 2\}$. Generalizing this approach, can one say something about the freeness by focusing on the analogue arrangements with edge set $E_n$ for general $n$? 
    \item As can be seen in Table \ref{tab:building sets_4}, there are different free arrangements with the same exponents. Is there a hypergraphic or building set theoretical explanation? 
\end{enumerate}

\printbibliography

\appendix 
\section{}\label{sec:appendix}

This table lists all free building set arrangements up until $n \leq 4$. It can be seen that there are more non-graphic free arrangements than graphic ones. Also, there are different arrangements with the same exponents. 
\begin{small}
\begin{longtable}{@{}llll p{1in}@{}}

\toprule
 \textsc{Building set} & \textsc{Exponents} & \textsc{$\mathcal{K(\CA)}$} & \textsc{Graphic}\\
 \midrule
 \{\{1\}\} & (1) & 2 & true ($P_1$) \\ 
 \{\{1\}, \{2\}, \{1, 2\}\} & (1,2) & 3 & true ($P_2$)\\ 
\{\{1\}, \{2\}, \{3\}, \{1, 2\}, \{1, 2, 3\}\} & (1,2,2)& 18 & false \\ 
 \{\{1\}, \{2\}, \{3\}, \{1, 2\}, \{2, 3\}, \{1, 2, 3\}\} & (1,2,3) & 24 & true ($P_3$) \\ 
\{\{1\}, \{2\}, \{3\}, \{1, 2\}, \{1, 3\}, \{2, 3\}, \{1, 2, 3\}\} & (1,3,3) & 32& true ($C_3$)\\ 
 \{\{1\},\{2\},\{3\},\{4\},\{1, 2\},\{3, 4\},\{1, 2, 3, 4\}\} & (1,2,2,2) & 54& false \\ 
 \{\{1\},\{2\},\{3\},\{4\},\{1, 2\},\{1, 2, 3\},\{1, 2, 3, 4\}\} &  (1,2,2,2) & 54& false \\ 
 \{\{1\},\{2\},\{3\},\{4\},\{1, 2\},\{3, 4\},\{1, 2, 3\},\{1, 2, 3, 4\}\} & (1,2,2,3) & 72& false \\
\{\{1\},\{2\},\{3\},\{4\},\{1, 2\},\{1, 2, 3\},\{1, 2, 4\},\{1, 2, 3, 4\}\} & (1,2,2,3) & 72& false \\ 
\{\{1\},\{2\},\{3\},\{4\},\{1, 2\},\{2, 3\},\{1, 2, 3\},\{1, 2, 3, 4\}\} & (1,2,2,3) & 72 & false \\ 
 \{\{1\},\{2\},\{3\},\{4\},\{1, 2\},\{3, 4\},\{1, 2, 3\},\{1, 2, 4\},\{1, 2, 3, 4\}\} & (1,2,3,3) &96& false \\ 
\{\{1\},\{2\},\{3\},\{4\},\{1, 2\},\{3, 4\},\{1, 2, 3\},\{1, 3, 4\},\{1, 2, 3, 4\}\} & (1,2,3,3) &96& false  \\ 
 \{\{1\},\{2\},\{3\},\{4\},\{1, 2\},\{2, 3\},\{1, 2, 3\},\{1, 2, 4\},\{1, 2, 3, 4\}\} & (1,2,3,3) &96& false \\ 
 \{\{1\},\{2\},\{3\},\{4\},\{1, 2\},\{1, 3\},\{2, 3\},\{1, 2, 3\},\{1, 2, 3, 4\}\} & (1,2,3,3) &96& false \\ 
 \{\{1\},\{2\},\{3\},\{4\},\{1, 2\},\{2, 3\},\{3, 4\},\{1, 2, 3\},\{2, 3, 4\},\{1, 2, 3, 4\}\} & (1,2,3,4) &120& true ($P_4$) \\ 
\{\{1\},\{2\},\{3\},\{4\},\{1, 2\},\{3, 4\},\{1, 2, 3\},\{1, 2, 4\},\{1, 3, 4\},\{1, 2, 3, 4\}\} & (1,3,3,3) &128& false \\ 
 \{\{1\},\{2\},\{3\},\{4\},\{1, 2\},\{1, 3\},\{2, 3\},\{1, 2, 3\},\{1, 2, 4\},\{1, 2, 3, 4\}\} & (1,3,3,3) &128& false \\ 
 \{\{1\},\{2\},\{3\},\{4\},\{1, 2\},\{2, 3\},\{1, 2, 3\},\{1, 2, 4\},\{2, 3, 4\},\{1, 2, 3, 4\}\} & (1,3,3,3) &128& false \\ 
 \{\{1\},\{2\},\{3\},\{4\},\{1, 2\},\{2, 3\},\{3, 4\},\{1, 2, 3\},\{1, 2, 4\},\{2, 3, 4\},\{1, 2, 3, 4\}\} & (1,3,3,4) &160& false \\ 
\{\{1\},\{2\},\{3\},\{4\},\{1, 2\},\{2, 3\},\{2, 4\},\{1, 2, 3\},\{1, 2, 4\},\{2, 3, 4\},\{1, 2, 3, 4\}\} & (1,3,3,4) &160& true ($A_{3, 2}$)\\ 
\{\{1\},\{2\},\{3\},\{4\},\{1, 2\},\{1, 3\},\{1, 4\},\{3, 4\},\{1, 2, 3\},\{1, 2, 4\},\{1, 3, 4\},\{1, 2, 3, 4\}\} & (1,3,4,4) &200& true ($\Delta_{3,2}$) \\ 
\{\{1\},\{2\},\{3\},\{4\},\{1, 2\},\{1, 3\},\{2, 4\},\{3, 4\},\{1, 2, 3\},\{1, 2, 4\},\{1, 3, 4\},\{2, 3, 4\},\{1, 2, 3, 4\}\} & (1,4,4,4) &250& true ($C_4$) \\ 

\bottomrule
\caption{Free Boolean building set arrangements for dimension up to 4.} 
\label{tab:building sets_4}
\end{longtable} 
\end{small}

\end{document}